\documentclass[12pt,reqno]{amsart} 
\usepackage{amssymb}
\usepackage{mathrsfs}
\usepackage{enumerate}
\usepackage[all]{xy}
\usepackage[usenames,dvipsnames]{color}
\usepackage[colorlinks=true, citecolor=OliveGreen, 
linkcolor=OliveGreen]{hyperref}

\renewcommand{\today}{\the\day/\the\month/\the\year}

\makeatletter
\@namedef{subjclassname@2020}{\textup{2020} Mathematics Subject 
Classification} 
\makeatother

\DeclareMathAlphabet\EuR{U}{eur}{m}{n}
\SetMathAlphabet\EuR{bold}{U}{eur}{b}{n}
\newcommand{\curs}{\EuR}
\newcommand{\catdef}[2][]{\expandafter\newcommand\csname#2\endcsname%
{#1\curs{#2}}}
\catdef{Ab}    \catdef{Grp}   \catdef{Cat}   \catdef{Vect}   \catdef{Mod}
\catdef{Is}    \catdef{Top}   \catdef{hoTop}  \catdef{ho}   \catdef{FibSimp}

\renewcommand{\mod}{\textup{-}\curs{mod}}

\let\Gamma=\varGamma
\let\Omega=\varOmega
\let\Sigma=\varSigma

\renewenvironment{enumerate}[1][]
{\begin{enumerat}[#1]\setlength{\itemsep}{6pt}}{\end{enumerat}}

\newenvironment{enuma}{\begin{enumerate}[{\rm(a) }]}{\end{enumerate}}
\newenvironment{enumi}{\begin{enumerate}[{\rm(i) }]}{\end{enumerate}}

\renewenvironment{itemize}
{\begin{itemiz}\setlength{\itemsep}{6pt} \setlength{\itemindent}{-5pt} }
{\end{itemiz}}

\definecolor{darkgreen}{rgb}{0,0.5,0}
\definecolor{bluegreen}{rgb}{0,0.2,0.8}
\definecolor{darkred}{rgb}{0.8,0,0}
\definecolor{newercolor}{rgb}{0.2,0,1}
\definecolor{darkyellow}{rgb}{0.7,0.7,0}
\definecolor{darkorange}{rgb}{0.8,0.4,0}

\newcommand{\mynote}[1]{{\color{blue}\noindent\textbf{\textup{[#1]}}}}

\numberwithin{table}{section}

\newlength{\short}
\newcommand{\dbl}[2]{%
\renewcommand{\arraystretch}{1.0}\renewcommand{\arraycolsep}{0pt}%
\begin{array}{r}\rule{0pt}{12pt}#1\\#2\end{array}}

\newcommand{\4}[1]{\widebar{#1}}
\newcommand{\5}[1]{\widehat{#1}}

\newcommand{\9}[1]{{}^{#1}\!}

\def\pair[#1,#2]{[\hskip-1.5pt[#1,#2]\hskip-1.5pt]}

\SelectTips{cm}{10} \UseTips   

\let\oldcirc=\circ
\renewcommand{\circ}{\mathchoice
    {\mathbin{\scriptstyle\oldcirc}}{\mathbin{\scriptstyle\oldcirc}}
    {\mathbin{\scriptscriptstyle\oldcirc}}
    {\mathbin{\scriptscriptstyle\oldcirc}}}

\def\beq#1\eeq{\begin{equation*}#1\end{equation*}}
\def\beqq#1\eeqq{\begin{equation}#1\end{equation}}

\numberwithin{equation}{section}

\newtheorem{Thm}{Theorem}[section]
\newtheorem{Prop}[Thm]{Proposition}
\newtheorem{Cor}[Thm]{Corollary}
\newtheorem{Lem}[Thm]{Lemma}

\newtheorem{Th}{Theorem}

\theoremstyle{definition}
\newtheorem{Defi}[Thm]{Definition}
\newtheorem{Ex}[Thm]{Example}

\newcommand{\widebar}[1]
      {\overset{{\mskip1mu\leaders\hrule height0.4pt\hfill\mskip1mu}}{#1}
      \vphantom{#1}}

\newcounter{let} 
\loop\stepcounter{let}
\expandafter\edef\csname cal\alph{let}\endcsname%
{\noexpand\mathcal{\Alph{let}}}
\ifnum\thelet<26\repeat

\loop\stepcounter{let}
\expandafter\edef\csname scr\alph{let}\endcsname%
{\noexpand\mathscr{\Alph{let}}}
\ifnum\thelet<26\repeat

\newcommand{\tdef}[2][]{\expandafter\newcommand\csname#2\endcsname%
{#1\textup{#2}}}
\tdef{Iso}   \tdef{Aut}    \tdef{Out}    \tdef{Inn}    \tdef{Hom}
\tdef{End}   \tdef{Inj}    \tdef{map}    \tdef{Ker}    \tdef{Ob}
\tdef{Mor}   \tdef{Res}    \tdef{Id}     \tdef{Fr}     \tdef{Spin} 
\tdef{rk}    \tdef{conj}   \tdef{incl}   \tdef{proj}   \tdef{diag} 
\tdef{trf}   \tdef{Sol}    \tdef{He}     \tdef{Sz}     \tdef{cj}
\tdef{Rep}   \tdef{pr}    \tdef{Inndiag} \tdef{Outdiag}  \tdef{expt}
\tdef{supp}  \tdef{Isom}   \tdef{ord}    \tdef{Coker}   \tdef{Tr}
\tdef[_]{typ} \tdef[^]{op} \tdef[^]{ab}   \tdef{lcm}  \tdef{McL}
\tdef{restr}  \tdef{Comp}  \tdef{HS}     \tdef{ev}    \tdef{Stab}
\tdef{srk}    \tdef{init}  \tdef{Max}    \tdef[_]{dpt}

\newcommand{\colim}{\mathop\textup{colim}\limits}

\newcommand{\Fix}[2]{\textup{Fix}(#1,#2)}   

\newcommand{\fdef}[1]{\expandafter\newcommand\csname#1\endcsname%
{\mathfrak{#1}}}
\fdef{X}  \fdef{red}  \fdef{foc}  \fdef{hyp}  \fdef{Lie} \fdef{Y} \fdef{ch}

\newcommand{\bbdef}[1]{\expandafter\newcommand%
\csname#1\endcsname{\mathbb{#1}}}
\bbdef{C} \bbdef{F} \bbdef{R} \bbdef{Z} \bbdef{N} \bbdef{Q} \bbdef{K}
\bbdef{D}

\newcommand{\itdef}[1]{\expandafter\newcommand\csname#1\endcsname%
{\textit{#1}}}
\itdef{PSL}  \itdef{PSU}  \itdef{SL}  \itdef{SU}  \itdef{GL} \itdef{GU}
\itdef{Sp}   \itdef{PSp} \itdef{PSO} \itdef{SO}   \itdef{SD} \itdef{PGU} 
\itdef{PGL}  \itdef{Co}  \itdef{Fi}  \itdef{GO}   \itdef{BDI}

\newcommand{\sminus}{\smallsetminus}
\newcommand{\lie}[3]{\def\test{#2}\def\tst{G}\ifx\test\tst{{}^{#1}#2_{#3}}
\else{{}^{#1}\!#2_{#3}}\fi}
\renewcommand{\*}{\,\lower6pt\hbox{\Large{\textup{*}}}\,}
\newcommand{\syl}[3][]{\textup{Syl}^{#1}_{#2}(#3)}
\newcommand{\sylp}[2][]{\syl[#1]{p}{#2}}

\renewcommand{\Im}{\textup{Im}}

\newcommand{\defeq}{\overset{\textup{def}}{=}}

\newcommand{\mxfoura}[8]{\left(\begin{smallmatrix}#1&#2&#3&#4\\#5&#6&#7&#8}
\newcommand{\mxfourb}[8]{\\#1&#2&#3&#4\\#5&#6&#7&#8\end{smallmatrix}\right)}

\renewcommand{\:}{\colon}
\newcommand{\pcom}{{}^\wedge_p}

\newcommand{\nsg}{\trianglelefteq}

\let\too=\longrightarrow
\let\xto=\xrightarrow

\let\fromm=\longleftarrow

\newcommand{\gen}[1]{{\langle}#1{\rangle}}

\newcommand{\longleft}[1]{\;{\leftarrow%
\count255=0 \loop \mathrel{\mkern-6mu}%
    \relbar\advance\count255 by1\ifnum\count255<#1\repeat}\;}
\newcommand{\longright}[1]{\;{\count255=0 \loop \relbar\mathrel{\mkern-6mu}%
    \advance\count255 by1\ifnum\count255<#1\repeat\rightarrow}\;}
\newcommand{\Right}[2]{\overset{#2}{\longright#1}}
\newcommand{\RIGHT}[3]{\mathrel{\mathop{\kern0pt\longright#1}
        \limits^{#2}_{#3}}}

\newcommand{\LEFT}[3]{\mathrel{\mathop{\kern0pt\longleft#1}\limits^{#2}_{#3}}
}
\newcommand{\dRIGHT}[3]{\mathrel{%
   \mathop{\vcenter{\baselineskip=0pt\hbox{$\kern0pt\longright#1$}%
   \hbox{$\kern0pt\longright#1$}}}\limits^{#2}_{#3}}}
\newcommand{\LRIGHT}[3]{\mathrel{%
   \mathop{\vcenter{\baselineskip=0pt\hbox{$\kern0pt\longleft#1$}%
   \hbox{$\kern0pt\longright#1$}}}\limits^{#2}_{#3}}}
\newcommand{\RLEFT}[3]{\mathrel{%
   \mathop{\vcenter{\baselineskip=0pt\hbox{$\kern0pt\longright#1$}%
   \hbox{$\kern0pt\longleft#1$}}}\limits^{#2}_{#3}}}
\newcommand{\onto}[1]{\;{\count255=0 \loop \relbar\mathrel{\mkern-6mu}%
    \advance\count255 by1
    \ifnum\count255<#1 \repeat \twoheadrightarrow}\;}

\newcommand{\dn}{{\downarrow}}

\newcommand{\longline}{\bigskip\hfill\hbox to 8cm{\hrulefill}%
\hfill\bigskip}

\def\LFS(#1){\textup{LFS($#1$)}} 
\def\LF(#1){\textup{LF($#1$)}} 

\fdef{Fin}
\fdef{I}
\tdef{Ext}

\newcommand{\higherlim}[2]{\displaystyle\setbox1=\hbox{\rm lim}
	\setbox3=\hbox{$\scriptstyle{#1}$}
	\ifdim\wd1<\wd3
	\mathop{\vtop{\baselineskip=5pt\box1}^{#2}}_{#1\hphantom{#2}}
	\else
	\mathop{\vtop{\baselineskip=5pt\box1}^{#2}}_{#1\hphantom{#2}}
	\fi}

\begin{document}

\title{Limits over orbit categories of locally finite groups}

\author{Bob Oliver}
\address{Universit\'e Sorbonne Paris Nord, LAGA, UMR 7539 du CNRS, 
99, Av. J.-B. Cl\'ement, 93430 Villetaneuse, France.}
\email{bobol@math.univ-paris13.fr}
\thanks{B. Oliver is partially supported by UMR 7539 of the CNRS}

\begin{abstract} 
We correct an error in the paper \cite{BLO3}, and take the opportunity to 
examine in more detail the derived functors of inverse limits over orbit 
categories of (infinite) locally finite groups. The main results show how 
to reduce this in many cases to limits over orbit categories of finite 
groups, but we also look at generalizations of the Lyndon-Hochschild-Serre 
spectral sequence for higher limits over orbit categories for an extension 
of locally finite groups. 
\end{abstract}




\maketitle

\bigskip

\setcounter{let}{0}
\loop\stepcounter{let}
\expandafter\edef\csname x\Alph{let}\endcsname%
{\noexpand\mathbf{\Alph{let}}}
\ifnum\thelet<26\repeat

\newcommand{\xd}{\textbf{\textit{d}}}

\section*{Introduction}

The orbit category of a group $G$ is the category $\calo(G)$ 
whose objects are the subgroups of $G$, and where a morphism from $U$ 
to $V$ is a $G$-map $G/U\too G/V$ between the corresponding orbits. (A 
different, but equivalent definition of the morphism sets is given in 
Definition \ref{d:orbit.cat}.) When $p$ is a 
prime, $\calo_p^f(G)\subseteq\calo_p(G)\subseteq\calo(G)$ denote the 
full subcategories whose objects are the finite $p$-subgroups, and all 
$p$-subgroups, respectively, of $G$. Here, by a $p$-group, we mean a 
group each of whose elements has (finite) order a power of $p$. 

When $G$ is a group and $M$ is a $\Z G$-module, let 
$F_M\:\calo_p(G)\op\too\Ab$ be the functor defined by setting $F_M(P)=0$ if 
$P\ne1$, and $F_M(1)=M$ with the given action of $\Aut_{\calo_p(G)}(1)\cong 
G$. Graded groups $\Lambda^*(G;M)\defeq\lim_{\calo_p(G)}^*(F_M)$ were 
defined by Jackowski, McClure, and Oliver in \cite{JMO} (at least for 
finite $G$), and they play an important role (as described in Proposition 
\ref{p:lim*=Lbda}) when describing right derived functors of inverse limits 
of arbitrary functors on $\calo_p(G)$. We refer to these higher derived 
functors as ``higher limits'' for short. 

These functors appeared again in work by Broto, Levi, and Oliver, mostly 
for finite groups, but also in \cite{BLO3} when $G$ is a (possibly 
infinite) locally finite group. Unfortunately, there was an error in the 
proof of \cite[Lemma 5.12]{BLO3} (see the discussion before Theorem 
\ref{t:5.12} below), and the original purpose of this paper was to correct 
that proof. Lemma 5.12 in \cite{BLO3} was needed to prove Theorems 8.7 and 
8.10 in the same paper, which say among other things that 
$|\call_S^c(G)|\pcom\simeq BG\pcom$ for every torsion group $G$ that is 
linear in characteristic different from $p$. We correct the proof of that 
lemma in Theorem \ref{t:5.12} in this paper (see also Theorem \ref{ThD}). 
The result in \cite{BLO3} for which this is needed is also restated here as 
Theorem \ref{t:blo3-8.7}. 

While fixing the error in \cite{BLO3}, we also take the opportunity to 
develop more of the theory of higher limits over orbit categories, and in 
particular the graded groups $\Lambda^*(G;M)$, for infinite groups $G$, 
especially when $G$ is locally finite (i.e., every finitely generated 
subgroup of $G$ is finite). For example, one of our main results is the 
following, where $\Fin(G)$ is the poset of finite subgroups of a group $G$.

\begin{Th} \label{ThA}
Fix a prime $p$, and let 
$G$ be a locally finite group. Then for each $\Z G$-module $M$, 
there is a first quarter spectral sequence
	\[ E_2^{ij} \cong \higherlim{\Fin(G)}i \bigl( \Lambda^j(-;M) \bigr) 
	\implies \Lambda^{i+j}(G;M) . \]
\end{Th}

Theorem \ref{ThA} follows from the spectral sequence of Theorem 
\ref{t:spseq.Lambda}, together with Theorem \ref{ThB}. Theorem 
\ref{t:spseq.Lambda} is a special case of Proposition \ref{p:spseq.Phi}, 
which involves higher limits of functors on orbit categories with more 
general sets of objects. Proposition \ref{p:spseq.Phi} is in turn a special 
case of the still more general Proposition \ref{p:spseq0}.

The next theorem is more technical, but is a key tool for many of the 
results in the paper. It is proven as Theorem \ref{t:Lb=Lbf}. Roughly, it 
says that $\Lambda^*$ is the same, whether we take limits with respect to 
all $p$-subgroups of $G$ or only the finite $p$-subgroups.

\begin{Th} \label{ThB}
Fix a prime $p$, and let $G$ be a group all of whose $p$-subgroups are 
locally finite. Then for each $\Z G$-module $M$,
	\[ \Lambda^*(G;M) \cong \Lambda^*_f(G;M) \defeq 
	\higherlim{\calo_p^f(G)}*(F_M). \]
\end{Th}

We also construct a spectral sequence that describes the behavior of the 
functors $\Lambda^*(-;-)$ under an extension of locally finite groups.

\begin{Th} \label{ThC}
Fix a prime $p$, a locally finite group $G$, a normal subgroup $H\nsg G$, 
and a $\Z G$-module $M$. Let $\chi\:G\too G/H$ be the natural map. Then 
there is a first quarter spectral sequence 
	\[ E_2^{ij} = \higherlim{\calo_p(G/H)}i(\Lambda^j(\chi^{-1}(-);M)) 
	\Longrightarrow \Lambda^{i+j}(G;M) . \] 
\end{Th}

Theorem \ref{ThC} is proven below as Theorem \ref{t:Lb.sp.seq}. It is a 
special case of Theorem \ref{t:G/H}: a more general spectral sequence 
involving higher limits over orbit spaces that also includes the 
Lyndon-Hochschild-Serre spectral sequence (see \cite[Theorem 
XI.10.1]{MacL-homol} 
or \cite[Theorem 6.8.2]{Weibel}) and Theorem \ref{ThB} as other special 
cases. 

Finally, we fill in the gap in \cite{BLO3} by proving the following 
result, shown below as Theorem \ref{t:5.12}.

\begin{Th} \label{ThD}
Fix a prime $p$. Let $G$ be a locally finite group, and let $M$ be an 
$\Z_{(p)}G$-module such that $C_G(M)$ contains an element of order $p$. 
Then $\Lambda^*(G;M)=0$. 
\end{Th}

Section \ref{s:background} contains some general results on categories 
and limits. In Section \ref{s:orbits}, we specialize to the case of orbit 
categories of groups, and construct among other things spectral sequences 
of which those in Theorems \ref{ThA} and \ref{ThC} are special cases. The 
functors $\Lambda^*(-;-)$ are then defined in Section \ref{s:Lambda}, where 
we study their basic properties and prove most of our main results, 
including Theorems \ref{ThA}, \ref{ThB}, \ref{ThC}, and \ref{ThD}. Then, in 
Section \ref{s:examples}, we give some examples to show how the functors 
$\Lambda^*(G;M)$ for locally finite groups $G$ can behave quite differently 
when $G$ is infinite than in the finite case. We end with an appendix, 
where conditions are given on a pair of $p$-groups $Q<P$ that imply that 
$Q<N_P(Q)$, and where we also construct an example to show that this is not 
always true, not even when the $p$-groups are locally finite.

\textbf{Notation:} Composition is always from right to left. By a 
\emph{$p$-group} (for a prime $p$), we always mean a group each of whose 
elements has $p$-power order. When $G$ is a group, we write
\begin{itemize} 

\item $\Fin(G)$ for the poset of finite subgroups of $G$; and 

\item $\9gH=gHg^{-1}$ and $H^g=g^{-1}Hg$ for $g\in G$ and $H\le G$.

\end{itemize}
Also, when $M$ is a $\Z G$-module and $H\le G$, we let $\Fix{H}M$ denote the 
fixed set of the action of $H$ on $M$. 

We will frequently regard a poset $(X,\le)$ as a category, where $X$ is the 
set of objects, and where there is a unique morphism $x\to y$ whenever $x\le y$.

We would like to thank Amnon Neeman and Dave Benson for their suggestions 
of references for higher limits over uncountable directed sets: a subject 
which is periferal to this paper but arose in connection with it.




\section{Background on categories and limits}
\label{s:background}

We begin by fixing some terminology. When $\calc$ is a small category, a 
\emph{$\calc$-module} is a functor $\calc\op\too\Ab$. We always work with 
contravariant functors, since those are the ones that appear most naturally 
in our work, but replacing them by covariant functors would mean only minor 
rephrasing of the definitions and results. 

More generally, when $R$ is a commutative ring, an \emph{$R\calc$-module} 
is a functor $\calc\op\too R\mod$. We let $\calc\mod$ and $R\calc\mod$ 
denote the categories of $\calc$-modules and $R\calc$-modules, 
respectively, where morphisms are the natural transformations of functors. 
(In earlier papers such as \cite{Mitchell}, these categories are often 
denoted $\Ab^{\calc\op}$ and $R\mod^{\calc\op}$, respectively, but we find 
the ``$\mod$'' notation natural and typographically simpler.)

If $\calc$ is a small category and $R$ is a commutative ring, then the 
category $R\calc\mod$ has enough 
injectives by Proposition \ref{p:I_c^A}(d) below. So the 
right derived functors of inverse limits of an $R\calc$-module $\Phi$ are 
defined, and we denote them $\lim_\calc^*(\Phi)$. Thus if $(0\to \Phi\to I_0\to 
I_1\to\cdots)$ is a resolution of $\Phi$ by injective $R\calc$-modules, then 
	\[ \higherlim\calc*(\Phi) = H^*\bigl( 0 \too \lim_\calc(I_0) \too 
	\lim_\calc(I_1) \too \lim_\calc(I_2) \too \cdots \bigr). \]
We usually refer to these derived functors of limits as the 
``higher limits'' of $\Phi$.

The following notation for certain injective or acyclic $\calc$-modules 
will often be useful.

\begin{Defi} \label{d:I_c^A}
Let $\calc$ be a small category, and let $R$ be a commutative ring. For 
each $c$ in $\calc$ and each $R$-module $M$, let 
$\I_{\calc,c}^M=\I_c^M$ be the 
$R\calc$-module that sends an object $d$ to 
	\[ \I_c^M(d) = \map(\Mor_\calc(c,d),M) \cong 
	\prod_{\Mor_\calc(c,d)}M. \]
A morphism $\varphi\in\Mor_\calc(d,d')$ induces a map of sets 
	$ \Mor_\calc(c,d)\too\Mor_\calc(c,d') $
via composition, and through that induces a homomorphism 
$\I_c^M(\varphi)\:\I_c^M(d')\too\I_c^M(d)$. 
\end{Defi}

The following result is well known and elementary, but we include a proof 
here since we've been unable to find a precise reference. The key ideas all 
go back at least to Mitchell \cite[p. 27]{Mitchell}, but he states his 
result under somewhat different assumptions.

\begin{Prop} \label{p:I_c^A}
Let $\calc$ be a small category, and let $R$ be a commutative ring.
\begin{enuma} 

\item For each $c$ in $\calc$, each $R$-module $M$, and each 
$R\calc$-module $\Phi$, $\Mor_{R\calc\mod}(\Phi,\I_c^M)\cong\Hom_R(\Phi(c),M)$. 
Hence $\I_c^M$ is an injective $R\calc$-module if $M$ is an injective 
$R$-module.

\item Every injective $R\calc$-module is a direct factor in a product of 
$R\calc$-modules $\I_c^M$, for objects $c$ in $\calc$ and injective $R$-modules 
$M$.

\item For each $c$ in $\calc$ and each $R$-module $M$ (injective or not), 
$\I_c^M$ is acyclic as an $R\calc$-module, in the sense that 
$\lim_\calc^i(\I_c^M)=0$ for all $i>0$.

\item The category $R\calc\mod$ has enough injectives.

\end{enuma}
\end{Prop}

\begin{proof} \noindent\textbf{(a) } Define homomorphisms 
	\[ \Mor_{R\calc\mod}(\Phi,\I_c^M) \RLEFT4{\Psi}{\Omega} 
	\Hom_R(\Phi(c),M) \]
as follows. Each $\alpha$ in $\Mor_{R\calc\mod}(\Phi,\I_c^M)$ induces a 
homomorphism $\alpha(c)\in\Hom_R(\Phi(c),\I_c^M(c))$ by evaluation at $c$, 
and we let $\Psi(\alpha)=\ev_{\Id_c}\circ\alpha(c)$ where $\ev_{\Id_c}$ from 
$\I_c^M(c)$ to $M$ is evaluation at $\Id_c$. 
Conversely, given $\beta\in\Hom_R(\Phi(c),M)$, define 
$\Omega(\beta)\in\Hom_{R\calc\mod}(\Phi,\I_c^M)$ by setting 
$\Omega(\beta)(d)(c\xto{\rho}d)=\beta\circ\Phi(\rho)\:\Phi(d)\too M$. That 
$\Psi\Omega(\beta)=\beta$ is immediate from these definitions, and the 
relation $\Omega\Psi(\alpha)=\alpha$ follows from the naturality properties 
of $\alpha$.

If $M$ is injective, and $\Phi_1\too\Phi_2$ is an injective morphism of 
$R\calc$-modules, then each morphism $\Phi_1\too\I_c^M$ extends to $\Phi_2$ 
since each homomorphism $\Phi_1(c)\too M$ extends to $\Phi_2(c)$. So 
$\I_c^M$ is injective as an $R\calc$-module.

\smallskip

\noindent\textbf{(b) } Fix $\Phi\:\calc\op\too R\mod$. For each 
$c\in\Ob(\calc)$, choose an injective $R$-module $I(c)$ and an injective 
homorphism $\chi(c)\:\Phi(c)\too I(c)$. By (a), there is an injective 
homomorphism of $R\calc$-modules from $\Phi$ into 
$\prod_{c\in\Ob(\calc)}\I_c^{I(c)}$. If $\Phi$ is injective, then this injection 
splits, and $\Phi$ is a direct factor in this product. 

\smallskip

\noindent\textbf{(c) } Let $0\too M\too I_0\too I_1\too\cdots$ be a 
resolution of $M$ by injective $R$-modules. Then by (a), 
$0\too\I_c^M\too\I_c^{I_0}\too\I_c^{I_1}\too\cdots$ is a 
resolution of $\I_c^M$ by injective $R\calc$-modules. For each $n\ge0$, 
	\[ \lim_{\calc}(\I_c^{I_n}) \cong 
	\Hom_{R\calc\mod}(\6R,\I_c^{I_n}) \cong \Hom_R(R,I_n)\cong I_n, 
	\]
where $\6R$ is the constant functor that sends all objects in $\calc$ to 
$R$, and where the second isomorphism holds by (a). So 
	\beq \higherlim{\calc}i(\I_c^M) \cong H^i( 0 \to I_0\to 
	I_1\to\cdots ) = \begin{cases} M & \textup{if $i=0$} \\
	0 & \textup{if $i\ge1$.} \end{cases} \eeq

\smallskip

\noindent\textbf{(d) } Let $\Phi$ be an $R\calc$-module. For each $c$ in 
$\calc$, choose an injective $R$-linear homomorphism $\psi_c\:\Phi(c)\too 
M_c$ where $M_c$ is injective in $R\mod$. (The category $R\mod$ has enough 
injectives by \cite[Theorem III.7.4]{MacL-homol}.) Set 
$\I=\prod_{c\in\Ob(\calc)}\I_c^{M_c}$: this is a product of injective 
$R\calc$-modules by (a) and hence is itself injective. By (a) again, there 
is an injective homomorphism $\Phi\too\I$, and thus $R\calc\mod$ has 
enough injectives.
\end{proof}

The ``bar resolution'' for higher limits over small categories will often 
be useful.

\begin{Prop} \label{p:C*(C;Phi)}
Let $\calc$ be a small category, and let $\Phi$ be a $\calc$-module. For 
each $n\ge0$, set
	\[ C^n(\calc;\Phi) = \prod_{c_0\to\cdots\to c_n} \Phi(c_0), \]
where the product is taken over all composable sequences of $n$ morphisms 
in $\calc$ (over all objects in $\calc$ if $n=0$). Define 
$d^n\:C^n(\calc;\Phi)\too C^{n+1}(\calc;\Phi)$ by setting, for $\xi\in 
C^n(\calc;\Phi)$, 
	\[ (d^n\xi)(c_0\xto{\chi}c_1\to\cdots\to c_{n+1}) = 
	\chi^*(\xi(c_1\to\cdots\to c_{n+1})) + \sum_{i=1}^{n+1}(-1)^i 
	\xi(c_0\to\cdots \5{c_i} \cdots\to c_{n+1}). \]
Here, ``$\5{c_i}$'' means that the term $c_i$ is removed from the sequence. Then 
there is a natural isomorphism
	\[ \higherlim{\calc}*(\Phi) \cong H^*\bigl( 0 \too C^0(\calc;\Phi) 
	\xto{~d^0~} C^1(\calc;\Phi) \xto{~d^1~} C^2(\calc;\Phi) \xto{~d^2~} 
	\cdots \bigr). \]
\end{Prop}

\begin{proof} This is shown in \cite[Appendix II, Proposition 3.3]{GZ}. A 
different proof, based on taking a projective resolution of the constant 
functor $\underline\Z$, is given in \cite[Lemma 2]{O-Steinb} and 
\cite[Proposition III.5.3]{AKO}. 
It can also be shown by constructing a resolution of $\Phi$ 
	\[ 0 \too \Phi \too \prod_{c_0}\I_{c_0}^{\Phi(c_0)} \too 
	\prod_{c_0\to c_1} \I_{c_1}^{\Phi(c_0)} \too 
	\prod_{c_0\to c_1\to c_2} \I_{c_2}^{\Phi(c_0)} \too \cdots  \]
by acyclic $\calc$-modules (Proposition \ref{p:I_c^A}(c)), and then 
applying Lemma \ref{l:lim*(C)=0} below to show that its homology after taking 
limits is isomorphic to $\higherlim{\calc}*(\Phi)$. 
\end{proof}

The next lemma gives us a tool in many cases for computing the homology of 
the limit of a chain complex of $\calc$-modules. It is particularly 
useful when $\calc$ is the category of a directed poset. 

\begin{Lem} \label{l:lim*(C)=0}
Let $\calc$ be a small category, let $R$ be a commutative ring, and let 
	\[ 0 \Right2{} \Phi^0 \Right3{d^0} \Phi^1 \Right3{d^1} \Phi^2 
	\Right3{d^2} \Phi^3 \Right3{d^3} \cdots \]
be a chain complex of $R\calc$-modules. Assume that each $\Phi^j$ is 
acyclic; i.e., that $\lim_{\calc}^i(\Phi^j)=0$ for all $i\ge1$ and all 
$j\ge0$. Set $\xH^j=H^j(\Phi^*,d^*)=\Ker(d^j)/\Im(d^{j-1})$ as an 
$R\calc$-module. Then there is a spectral sequence 
	\[ E_2^{ij} = \higherlim{\calc}i\bigl( \xH^j \bigr) 
	~\Longrightarrow~ 
	H^{i+j}\bigl(\lim_\calc(\Phi^*),\lim_\calc(d^*)\bigr). \]
\end{Lem}

\begin{proof} Set $Z^j=\Ker(d^j)$ and $B^{j+1}=\Im(d^j)$ for each 
$j\ge0$ (and set $B^0=0$), all regarded as $R\calc$-modules. 
Thus there are short exact sequences of $R\calc$-modules
	\beqq \begin{split} 
	&0 \Right2{} B^j \Right3{} Z^j \Right3{} \xH^j \Right2{} 0 \\
	&0 \Right2{} Z^j \Right3{} \Phi^j \Right3{} B^{j+1} \Right2{} 0 
	\end{split} \label{e:lim*(C)=0} \eeqq
for all $j\ge0$. For each $j\ge0$, choose injective resolutions 
	\begin{align*} 
	&0 \Right2{} B^j \Right3{} I^{0j}_{(B)} \Right3{} I^{1j}_{(B)}
	\Right3{} I^{2j}_{(B)} \Right3{} \cdots \\
	&0 \Right2{} \xH^j \Right3{} I^{0j}_{(H)} \Right3{} I^{1j}_{(H)}
	\Right3{} I^{2j}_{(H)} \Right3{} \cdots .
	\end{align*}
By the horseshoe lemma (see \cite[Proposition 6.5]{Osborne}, or 
\cite[Proposition I.3.5]{CE} for the dual version), there are injective 
resolutions 
	\begin{align*} 
	&0 \Right2{} Z^j \Right3{} I^{0j}_{(Z)} \Right3{} I^{1j}_{(Z)}
	\Right3{} I^{2j}_{(Z)} \Right3{} \cdots \\
	&0 \Right2{} \Phi^j \Right3{} I^{0j}_{(\Phi)} \Right3{} I^{1j}_{(\Phi)}
	\Right3{} I^{2j}_{(\Phi)} \Right3{} \cdots 
	\end{align*}
which fit into short exact sequences of resolutions 
	\beqq \begin{split} 
	&0 \Right2{} I^{*j}_{(B)} \Right3{\alpha^{*j}} I^{*j}_{(Z)} 
	\Right3{\beta^{*j}} 
	I^{*j}_{(H)} \Right2{} 0 \\
	&0 \Right2{} I^{*j}_{(Z)} \Right3{\gamma^{*j}} I^{*j}_{(\Phi)} 
	\Right3{\delta^{*j}} 
	I^{*,j+1}_{(B)} \Right2{} 0 
	\end{split} \label{e:lim*(C)=0-2} \eeqq
of the sequences in \eqref{e:lim*(C)=0}.

Now consider the sequence of injective resolutions 
	\[ 0 \Right2{} I^{*0}_{(\Phi)} \Right3{} I^{*1}_{(\Phi)} 
	\Right3{} I^{*2}_{(\Phi)} \Right3{} \cdots \]
where each morphism of resolutions is the composite 
	\[ I^{*j}_{(\Phi)} \Right3{\delta^{*j}} I^{*,j+1}_{(B)} 
	\Right4{\alpha^{*,j+1}} 
	I^{*,j+1}_{(Z)} \Right4{\gamma^{*,j+1}} I^{*,j+1}_{(\Phi)}. \]
We regard this as a double complex $\{I^{ij}_{(\Phi)}\}_{i,j\ge0}$ of 
injective objects in $\calc\mod$, where the $j$-th row is the 
injective resolution $I^{*j}_{(\Phi)}$ of $\Phi^j$. Set 
$X^{ij}=\lim_\calc(I^{ij}_{(\Phi)})$. 

Consider the two spectral sequences induced by the double complex 
$X^{ij}$: let $\5E$ be that obtained by taking homology first of the 
rows and then of the columns, and let $E$ be the other one. 
Since the $j$-th row $I^{*j}_{(\Phi)}$ is an injective resolution of 
$\Phi^j$, we have $\5E_1^{ij}\cong\lim_\calc^i(\Phi^j)$. Thus for all 
$j\ge0$, we have $\5E_1^{0j}\cong\lim_\calc(\Phi^j)$, while 
$\5E_1^{ij}=0$ for $i>0$ by assumption. So $\5E$ collapses, and $E$ and 
$\5E$ both converge to $H^*(\lim_{\calc}(\Phi^*),\lim_{\calc}(d^*))$. 

The short exact sequences of injectives \eqref{e:lim*(C)=0-2} are still 
exact after taking limits over $\calc$. So the homology of the $i$-th 
column $X^{i*}$ is $E_1^{i*}\cong\lim_{\calc}(I^{i*}_{(H)})$. Thus the 
$j$-th row in the $E_1$-term is $E_1^{*j}\cong\lim_\calc(I^{*j}_{(H)})$, where 
$\{I^{*j}_{(H)}\}_{j\ge0}$ is an injective resolution of $\xH^j$. It 
follows that $E_2^{ij}\cong \lim_{\calc}^i(\xH^j)$. 
\end{proof}

The following definition of a directed category is that used by Bass 
\cite[p. 44]{Bass}. 


\begin{Defi} \label{d:directed}
A category $\calc$ is \emph{directed} if it satisfies the following 
conditions:
\begin{enuma} 
\item For each pair of objects $c_1,c_2\in\Ob(\calc)$, there is a third 
object $d$ in $\calc$, together with morphisms $c_1\too d\fromm c_2$.
\item For each pair of objects $c_1,c_2\in\Ob(\calc)$ and each pair of 
morphisms $\varphi,\psi\in\Mor_\calc(c_1,c_2)$, there is an object $d$ in 
$\calc$ and $\chi\in\Mor_\calc(c_2,d)$ such that $\chi\varphi=\chi\psi$.
\end{enuma}
\end{Defi}

Note that a poset is directed if and only if its category is directed. 
(Condition (b) always holds for the category of a poset, since there is at 
most one morphism between any pair of objects.)

We recall here the definition of over- and undercategories, as well as Kan 
extensions. Let $F\:\calc\too\cald$ be a functor between small categories. 
For each object $d$ in $\cald$, let $d\dn F$ (the undercategory) and $F\dn 
d$ (the overcategory) be the categories with objects
	\begin{align*} 
	\Ob(d\dn F) &= \bigl\{ (c,\varphi) \,\big|\, c\in\Ob(\calc),~ 
	\varphi\in\Mor_\cald(d,F(c)) \bigr\} \\
	\Ob(F\dn d) &= \bigl\{ (c,\varphi) \,\big|\, c\in\Ob(\calc),~ 
	\varphi\in\Mor_\cald(F(c),d) \bigr\}  . 
	\end{align*}
A morphism in $d\dn F$ from $(c_1,\varphi_1)$ to $(c_2,\varphi_2)$ is a 
morphism $\rho\in\Mor_\calc(c_1,c_2)$ such that 
$\varphi_2=F(\rho)\circ\varphi_1$, and similarly for $F\dn d$.

If $F\:\calc\too\cald$ is a functor between small categories and 
$\Phi\:\calc\op\too R\mod$ is an $R\calc$-module, then the left and right 
Kan extensions along $F$ are functors $F_*^L,F_*^R\:R\calc\mod\too 
R\cald\mod$ that are left and right adjoint, respectively, to the functor 
sending $\Psi$ to $\Psi\circ F$. They are described explicitly by the 
formulas 
	\begin{align} 
	(F_*^L\Psi)(d) &= \colim_{d{\downarrow}F}\bigl( 
	(d{\downarrow}F)\op \Right3{\mu\op} \calc\op \Right3{\Psi} R\mod 
	\bigr) \label{e:F^L_*} \\
	(F^R_*\Psi)(d) &= \lim_{F{\downarrow}d}\bigl( 
	(F{\downarrow}d)\op \Right3{\mu\op} \calc\op \Right3{\Psi} R\mod 
	\bigr), \label{e:F^R_*}
	\end{align}
where in both cases, $\mu$ is the forgetful functor sending an object 
$(c,\varphi)$ to $c$. We refer to \cite[\S\, X.3]{MacL-categ} for the description 
of $F_*^R\Phi$ (but note that we assume that $\Phi$ is contravariant). The 
formula for $F_*^L\Phi$ can then be obtained by replacing each category by 
its opposite.

We next note some conditions under which the restriction of a functor to a 
subcategory has the same higher limits. 

\begin{Lem} \label{l:lim*(C0)}
Let $F\:\calc\too\cald$ be a functor between small categories. Assume, for 
each $d\in\Ob(\cald)$, that $(d{\downarrow}F)\op$ is nonempty and directed. 
Then for each commutative ring $R$ and each functor $\Phi\:\cald\op\to 
R\mod$, we have $\higherlim{\cald}q(\Phi)\cong \higherlim{\calc}q(\Phi\circ 
F)$ for all $q\ge0$. 
\end{Lem}

\begin{proof} For each $d\in\Ob(\cald)$, the undercategory $d{\downarrow}F$ 
is nonempty by assumption, and is connected since it is directed (Definition 
\ref{d:directed}(a)). So by \cite[\S\,IX.3, Theorem 1]{MacL-categ}, 
$\lim_\cald(\Phi)\cong\lim_{\calc}(\Phi\circ F)$ for each $R\cald$-module 
$\Phi$. 

Fix an $R\cald$-module $\Phi$, and let $0\too\Phi\too I_0\too I_1\too\cdots$ be 
a resolution of $\Phi$ by injective $R\cald$-modules. The sequence 
	\[ 0 \Right3{} \Phi\circ F \Right3{} I_0\circ F 
	\Right3{} I_1\circ F \Right3{} I_2\circ F \Right3{} 
	\cdots \]
is an exact sequence of $R\calc$-modules, and we just saw that 
$\lim_{\calc}(I_n\circ F)\cong \lim_{\cald}(I_n)$ for each $n$. So if 
$I_n\circ F$ is injective for each $n$, then 
$\lim_{\calc}^*(\Phi\circ F)\cong\lim_{\cald}^*(\Phi)$, which 
is what we want to show.

It thus remains to prove that composition with $F$ sends injective 
$R\cald$-modules to injective $R\calc$-modules. To see this, it suffices to 
show that left Kan extension along $F$ sends injective maps of 
$R\calc$-modules to injective maps of $R\cald$-modules (since it is left 
adjoint to composition with $F$). By \eqref{e:F^L_*}, it suffices to show 
that colimits of functors on $(d{\downarrow}F)\op$ are left exact, and this 
follows from \cite[Theorem 2.6.15]{Weibel}, applied with 
$(d{\downarrow}F)\op$ in the role of $I$. (What Weibel calls a filtering 
category is what we call here a nonempty directed category.)
\end{proof}

The next proposition is an application of Lemma \ref{l:lim*(C)=0}, and is 
useful in certain cases when comparing higher limits over a category to 
those over a family of subcategories.

\begin{Prop} \label{p:spseq0}
Fix a commutative ring $R$. Let $\calc$ be a nonempty small category, and 
let $\D$ be a set of subcategories of $\calc$, regarded as a poset via 
inclusion. Assume that
\begin{enumi} 

\item $\calc$ is the union of the subcategories in $\D$; and 

\item $\D$ is a directed poset.

\end{enumi}
Then for each $R\calc$-module $\Phi$, there is a first quarter spectral 
sequence of $R$-modules
	\[ E_2^{ij} \cong \higherlim{\D}i 
	\bigl( \cald \mapsto \higherlim{\cald}j(\Phi|_\cald) \bigr) 
	\implies \higherlim{\calc}{i+j}(\Phi). \]
\end{Prop}

\begin{proof} For each $n\ge0$ and each $\cald\subseteq\calc$, let 
$\ch_n(\cald)$ be the set of all sequences $c_0\to\cdots\to c_n$ of objects 
and morphisms in $\cald$. For each $\sigma=(c_0\to\cdots\to 
c_n)\in\ch_n(\calc)$, set $\init(\sigma)=c_0$. By (i) and (ii), 
$\ch_n(\calc)$ is the union of the $\ch_n(\cald)$ for all $\cald\in\D$. So 
in the notation of Proposition \ref{p:C*(C;Phi)}, for each $n\ge0$, 
	\[ C^n(\calc;\Phi) = \prod_{\sigma\in\ch_n(\calc)} \Phi(\init(\sigma)) 
	\cong \lim_{\cald\in\D} \Bigl( \prod_{\sigma\in\ch_n(\cald)} 
	\Phi(\init(\sigma)) \Bigr) 
	\cong \lim_{\cald\in\D}\bigl(C^n(\cald;\Phi|_\cald)\bigr) \]
Hence the proposition follows from Lemma \ref{l:lim*(C)=0} once we show 
that each of the functors $(\cald\mapsto C^n(\cald;\Phi|_\cald))$ is acyclic 
as an $R\D$-module.

For each $\sigma\in\ch_n(\calc)$, let $X(\sigma)$ be the set of all 
members of $\D$ that contain $\sigma$ (which we just saw is nonempty). 
For each $\sigma\in\ch_n(\calc)$ and each abelian group $A$, let 
$\Omega_{X(\sigma)}^A$ 
be the $\D$-module that sends each member of $X(\sigma)$ to $A$ and each 
inclusion between them to $\Id_A$, and sends all members of 
$\D\sminus X(\sigma)$ to $0$. Then as $\D$-modules,
	\beqq \Bigl( \cald \mapsto C^n(\cald;\Phi|_\cald) \Bigr) \cong 
	\prod_{\sigma\in\ch_n(\calc)}\Omega_{X(\sigma)}^{\Phi(\init(\sigma))}. 
	\label{e:spseq0a} \eeqq
Also, for each $\sigma$ and $A$, we have $\lim_{\D}^*(\Omega_{X(\sigma)}^A)
\cong \lim_{X(\sigma)}^*(\6A)$, where $\6A$ denotes the constant functor, by Proposition 
\ref{p:C*(C;Phi)} and since $C^*(\D;\Omega_{X(\sigma)}^A) \cong 
C^*(X(\sigma);\6A)$. 
So it remains to show, for each $\sigma\in\ch_n(\calc)$, that every 
constant functor $X(\sigma)\op\too R\mod$ is acyclic. The bar 
resolution for a constant functor is the same up to isomorphism whether we 
regard it as a functor on $X(\sigma)\op$ or on $X(\sigma)$, so we will be 
done upon showing that every constant functor $X(\sigma)\too R\mod$ is 
acyclic. 

Let $*$ be the category with one object $o$ and the identity morphism, and 
let $F\:X(\sigma)\too*$ be the (unique) functor. Fix an $R$-module $M$, 
and let $\Phi_M\:*\too R\mod$ be the functor that sends $o$ to $M$. Then 
$o{\downarrow}F\cong X(\sigma)$, we already showed that it is nonempty, and 
it is directed by (ii). So by Lemma \ref{l:lim*(C0)}, applied with 
$F\op\:X(\sigma)\op\too*$ in the role of $F\:\calc\too\cald$, 
the constant functor $\6M=\Phi_M\circ F$ is acyclic.
\end{proof}


\section{Limits over orbit categories}
\label{s:orbits}

In this section, we first give the definition and some of the basic 
properties of orbit categories of groups, and then construct two 
spectral sequences that involve higher limits over orbit categories. 
We are mostly interested in orbit categories of locally finite groups, 
but many of these results hold just as easily for arbitrary groups.

\begin{Defi} \label{d:orbit.cat}
Let $G$ be a group. 
\begin{enuma} 

\item Let $\calo(G)$ be the \emph{orbit category} of $G$: 
the category whose objects are the subgroups of $G$, and where for each 
$H,K\le G$,
	\[ \Mor_{\calo(G)}(H,K) = K{\backslash}T_G(H,K) 
	= \{K g \,|\, g\in G,~ \9gH\le K \}. \] 
We identify $\Mor_{\calo(G)}(H,K)\cong\map_G(G/H,G/K)$, where a morphism 
$Kg$ sends $xH\in G/H$ to $xg^{-1}K\in G/K$.

\item When $X$ is a nonempty set of subgroups of $G$ invariant under 
conjugation, we let $\calo_X(G)\subseteq\calo(G)$ denote the full 
subcategory whose objects are the members of $X$. 


\item As a special case of (b), when $p$ is a prime, 
$\calo_p^f(G)\subseteq\calo_p(G)\subseteq\calo(G)$ denote the full 
subcategories whose objects are the finite $p$-subgroups and arbitrary 
$p$-subgroups of $G$, respectively. 

\end{enuma}
\end{Defi}

When it does not cause confusion, we write $[g]=Kg$ to denote a morphism in 
$\calo(G)$ induced by $g\in G$.

A morphism $f\:a\to b$ in a category $\calc$ is an \emph{epimorphism} (in 
the categorical sense) if $gf=hf$ implies $g=h$ for each pair of morphisms 
$g,h\:b\to c$. It will be useful to know that morphisms in an orbit 
category are epimorphisms.

\begin{Lem} \label{l:epi}
For every group $G$, all morphisms in $\calo(G)$, and hence all morphisms 
in every subcategory of $\calo(G)$, are epimorphisms in the categorical 
sense.
\end{Lem}

\begin{proof} Let $Kg\in\Mor_{\calo_X(G)}(H,K)$ and 
$Lx,Ly\in\Mor_{\calo_X(G)}(K,L)$ be such that $Lx\circ Kg = Ly\circ Kg$ 
(for some $H,K,L\le G$). Then $Lxg=Lyg$, so $Lx=Ly$ as cosets and hence as 
morphisms in $\calo_X(G)$. Thus $Kg$ is an epimorphism.
\end{proof}

We next show a version of Lemma \ref{l:lim*(C0)} specialized to orbit 
categories. 

\begin{Lem} \label{l:X0<X}
Let $G$ be a group, and let $X_0\subseteq X$ be nonempty sets of subgroups 
of $G$, both invariant under conjugation. Assume that 
\begin{enumi} 

\item each member of $X$ is contained in a member of $X_0$, and

\item $X_0$ is closed under finite intersections.

\end{enumi}
Then for each $\calo_X(G)$-module $\Phi$, 
	\beqq \higherlim{\calo_X(G)}*\!\!(\Phi) \cong 
	\higherlim{\calo_{X_0}(G)}*\!\!(\Phi|_{\calo_{X_0}(G)}). 
	\label{e:X0<X} \eeqq
\end{Lem}

\begin{proof} We apply Lemma \ref{l:lim*(C0)}, with the inclusion functor 
$\cali\:\calo_{X_0}(G)\too\calo_X(G)$ in the role of $F$. For each $H\in 
X$, the undercategory $H{\downarrow}\cali$ is nonempty by (i), and it 
satisfies condition (a) in the definition of a directed category by (ii). 
Since all morphisms in $\calo_X(G)$ are epimorphisms by Lemma \ref{l:epi}, 
there is at most one morphism between any given pair of objects in 
$H{\downarrow}\cali$, and so condition (b) in Definition \ref{d:directed} 
also holds. Thus $H{\downarrow}\cali$ is directed, and \eqref{e:X0<X} follows 
from Lemma \ref{l:lim*(C0)}. 
\end{proof}

We now turn to some spectral sequences involving orbit categories.  

As usual, we say that a group $G$ is \emph{locally finite} if every 
finitely generated subgroup of $G$ is finite. In general, orbit spaces of 
locally finite groups are easier to work with than those of arbitrary 
infinite groups. For example, in a locally finite group $G$, the poset 
$\Fin(G)$ of its finite subgroups is a directed poset, and this is 
important when we want to describe limits over $\calo_p(G)$ in terms of 
finite subgroups of $G$. 

\begin{Prop} \label{p:spseq.Phi}
Let $G$ be a locally finite group, and let $X$ be a nonempty set of finite 
subgroups of $G$ invariant under conjugation. For each $K\le G$, let $X\cap K$ 
be the set of members of $X$ that are contained in $K$. Then for each 
commutative ring $R$ and each $R\calo_X(G)$-module $\Phi$, there is a first 
quarter spectral sequence 
	\beqq E_2^{ij} \cong \higherlim{K\in\Fin(G)}i\bigl( 
	\higherlim{\calo_{X\cap K}(K)}j(\Phi|_{\calo_{X\cap K}(K)}) \bigr) 
	\Longrightarrow \higherlim{\calo_X(G)}{i+j}(\Phi) 
	\label{e:spseq.Phi} \eeqq
of $R$-modules. In particular, 
	\[ \lim_{\calo_X(G)}(\Phi) \cong 
	\lim_{K\in\Fin(G)}\bigl(\lim_{\calo_{X\cap K}(K)}
	(\Phi|_{\calo_{X\cap K}(K)})\bigr), \]
and if $G$ is countable, then for each $n\ge1$ there is a short exact sequence 
	\begin{multline*} 
	0 \Right3{} \higherlim{K\in\Fin(G)}1\bigl(\higherlim{\calo_{X\cap 
	K}(K)}{n-1}(\Phi|_{\calo_{X\cap K}(K)})\bigr) \Right3{} 
	\higherlim{\calo_X(G)}n(\Phi) \Right3{} \\
	\lim_{K\in\Fin(G)}\bigl(\higherlim{\calo_{X\cap K}(K)}n
	(\Phi|_{\calo_{X\cap K}(K)})\bigr) \Right3{} 0. 
	\end{multline*}
\end{Prop}

\begin{proof} Let $\D$ be the poset of all subcategories $\calo_{X\cap 
K}(K)\subseteq \calo_X(G)$ for $K\in\Fin(G)$, and let $F\:\Fin(G)\too\D$ be 
the surjective functor that sends $K$ to $\calo_{X\cap K}(K)$. Then 
$\calo_X(G)$ is the union of the members of $\D$ since all members of $X$ 
are finite and $G$ is locally finite, and $\Fin(G)$ and $\D$ are directed 
posets since $G$ is locally finite. So by Proposition \ref{p:spseq0}, for 
each $R\calo_X(G)$-module $\Phi$, there is a first quarter spectral sequence 
	\beqq E_2^{ij} \cong \higherlim{\D}i \bigl( \cald\mapsto 
	\higherlim{\cald}j(\Phi|_\cald) \bigr) \implies 
	\higherlim{\calo_X(G)}{i+j}(\Phi). \label{e:spseq.Phi2} \eeqq

For each $\cald\in\D$, the undercategory $\cald\dn F$ is isomorphic to the 
poset of all $H\in\Fin(G)$ such that $\calo_{X\cap H}(H)\ge \cald$, and 
hence is nonempty and directed. So by Lemma \ref{l:lim*(C0)}, for each 
$R\D$-module $\Psi$, $\higherlim{\Fin(G)}*(\Psi\circ F)\cong 
\higherlim{\D}*(\Psi)$. In particular, \eqref{e:spseq.Phi} follows from 
\eqref{e:spseq.Phi2}.

If $G$ is countable, then $\Fin(G)$ is a countable directed poset, so 
higher limits over $\Fin(G)$ vanish in degrees greater than $1$, and the 
$E_2$-term of the spectral sequence vanishes except for the first two 
columns.
\end{proof}

More generally, if $G$ has cardinality $\aleph_m$ for some finite $m\ge0$, 
then $\Fin(G)$ also has cardinality $\aleph_m$. So by a theorem of Goblot 
\cite[Proposition 2]{Goblot} (see also \cite[Th\'eor\`eme 3.1]{Jensen}), in 
the spectral sequence of Proposition \ref{p:spseq.Phi}, the terms 
$E_2^{ij}$ are always zero for $i\ge m+2$.

We next construct a spectral sequence that describes higher limits of 
functors on orbit categories for a group extension. We refer to 
\eqref{e:F^R_*} and the discussion before that for the definition and 
description of right Kan extensions.

\begingroup

\begin{Thm} \label{t:G/H}
Let $R$ be a commutative ring, 
let $G$ be a (discrete) group, let $H\nsg G$ be a normal subgroup, and let 
$\chi\:G\too G/H$ be the natural homomorphism. Let $X$ and $Y$ be nonempty 
sets of subgroups of $G$ and $G/H$, respectively, both invariant under 
conjugation, and such that $K\in X$ implies $\chi(K)\in Y$. Let 
$\5\chi\:\calo_X(G)\too\calo_Y(G/H)$ be the functor that sends $K\in X$ to 
$\chi(K)$ and sends a morphism $[g]$ to $[\chi(g)]$, and let 
	\[ \5\chi^*\: R\calo_Y(G/H)\mod \Right4{} R\calo_X(G)\mod \]
be the functor that sends $\Phi$ to $\Phi\circ\5\chi$. 
Let $\5\chi^R_*\:R\calo_X(G)\mod\too R\calo_Y(G/H)\mod$ be 
the right Kan extension along $\5\chi$. 
\begin{enuma} 

\item For each $R\calo_X(G)$-module $\Phi$, there is a first 
quarter spectral sequence 
	\[ E_2^{ij} = \higherlim{\calo_Y(G/H)}i 
	\bigl((R^j\5\chi^R_*)(\Phi)\bigr) 
	\Longrightarrow \higherlim{\calo_X(G)}{i+j}(\Phi), \]
where $R^j\5\chi^R_*\:R\calo_X(G)\mod\too R\calo_Y(G/H)\mod$ is the 
$j$-th right derived functor of $\5\chi^R_*$. 

\item For each $R\calo_X(G)$-module $\Phi$, each subgroup 
$K/H\in Y$, and each $j\ge0$,
	\[ (R^j\5\chi^R_*)(\Phi)(K/H)\cong 
	\higherlim{\calo_{X\cap K}(K)}j
	(\Phi|_{\calo_{X\cap K}(K)}), \]
where $X\cap K$ is the set of members of $X$ contained in $K$. Under this 
identification, a morphism $[g]\:K_1/H\too K_2/H$ in $\calo_Y(G/H)$ (where 
$g\in G$ and $\9gK_1\le K_2$) is sent to the homomorphism
	\[ \higherlim{\calo_{X\cap K_2}(K_2)}j (\Phi|_{\calo_{X\cap K_2}(K_2)}) 
	\Right3{} \higherlim{\calo_{X\cap K_1}(K_1)}j (\Phi|_{\calo_{X\cap 
	K_1}(K_1)}) \]
induced by the functor $c_g^{K_1}\:\calo_{X\cap K_1}(K_1)\too\calo_{X\cap 
K_2}(K_2)$ that sends $L\in X\cap K_1$ to $\9gL\in X\cap K_2$ and sends a 
morphism $[x]$ to $[\9gx]$, together with the natural transformation of 
functors
	\[ \Phi|_{\calo_{X\cap K_2}(K_2)} \circ c_g^{K_1} \Right4{} 
	\Phi|_{\calo_{X\cap K_1}(K_1)} \]
that sends an object $L$ in $\calo_{X\cap K_2}(K_2)$ to 
$\Phi([g])\in\Hom_R(\Phi(\9gL),\Phi(L))$.

\end{enuma}
\end{Thm}

\begin{proof} Set $\4G=G/H$ for short. For each $K\le G$, set 
$\4K=\chi(K)=K H/H$. 

By \eqref{e:F^R_*}, for each $R\calo_X(G)$-module 
$\Phi$ and each $\4K\in Y$, 
	\[ \5\chi^R_*(\Phi)(\4K) = 
	\lim_{\5\chi{\downarrow}\4K} \bigl((\5\chi{\downarrow}\4K)\op 
	\Right3{\mu\op} \calo_X(G)\op \Right3{\Phi} R\mod \bigr), \]
where $\mu$ is the forgetful functor sending $(L,\4L\to\4K)$ 
to $L$. Since inverse limits are left exact, so is $\5\chi^R_*$. 

\smallskip

\noindent\textbf{(a) } This is a special case of the Grothendieck spectral 
sequence, in the form described in \cite[Theorem 5.8.3]{Weibel}, applied to 
the triangle
	\[ \vcenter{\xymatrix@C=30pt@R=25pt{ 
	\calo_X(G)\mod \ar[rr]^{\5\chi^R_*} \ar[dr]_{\lim} 
	&& \calo_Y(\4G)\mod \ar[dl]^{\lim} \\
	& R\mod }} \]
of categories and functors. We already saw that $\5\chi^R_*$ is left exact, and 
limits are always left exact. So to be able to apply the spectral sequence 
(with $\5\chi_*^R$ and $\textup{lim}$ in the roles of $G$ and $F$), 
it remains to check that the triangle commutes up to natural isomorphism, 
and that $\5\chi^R_*$ sends injectives to injectives. This last condition 
(sending injectives to injectives) holds by \cite[Proposition 
2.3.10]{Weibel} and since $\5\chi^R_*$ is right adjoint to the exact functor 
$\5\chi^*$.

Let $\6R$ denote the constant functor on $\calo_Y(\4G)$ that sends all 
objects to $R$ and all morphisms to $\Id_R$. For each $R\calo_X(G)$-module 
$\Phi$, there are natural isomorphisms 
	\[ \lim_{\calo_Y(\4G)}(\5\chi^R_*(\Phi)) \cong 
	\Mor_{R\calo_Y(\4G)\mod}(\6R,\5\chi^R_*(\Phi)) 
	\cong \Mor_{R\calo_X(G)\mod}(\5\chi^*(\6R),\Phi) \cong 
	\lim_{\calo_X(G)}(\Phi), \]
where the second holds since $\5\chi^R_*$ is right adjoint to $\5\chi^*$. So 
the triangle commutes up to natural isomorphism.

\smallskip

\noindent\textbf{(b) } Let $0 \too \Phi \too I_0 \too I_1 \too I_2 \too 
\cdots$ be a resolution of $\Phi$ by injective $\calo_X(G)$-modules. Thus 
$(R^*\5\chi^R_*)(\Phi)$ is the homology of the complex of 
$R\calo_Y(\4G)$-modules
	\beqq 0 \too \5\chi^R_*(I_0) \Right3{} \5\chi^R_*(I_1) \Right3{}
	\5\chi^R_*(I_2) \Right3{} \cdots. \label{e:Rqk(FM)-1} \eeqq

We first claim that for each subgroup $K\le G$, 
	\beqq \textup{$\Phi$ injective in $\calo_X(G)\mod$ $\implies$ 
	$\Phi|_{\calo_{X\cap K}(K)}$ injective in 
	$\calo_{X\cap K}(K)\mod$,} 
	\label{e:Rqk(FM)-2} \eeqq
where $X\cap K$ is the set of members of $X$ contained in $K$. 
Recall (Proposition \ref{p:I_c^A}) that each injective object in 
$\calo_X(G)\mod$ is a direct factor in a product of injectives of the form 
$\I_L^A$ for $L\in X$ and $A$ an injective abelian group. So it 
suffices to prove that the restriction of each such $\I_L^A$ is injective. 
Fix a set $W\subseteq G$ of right coset representatives for $K$ in $G$ 
(thus $G=KW$). For given $L$ and $A$, we have 
	\[ \I_{\calo_X(G),L}^A|_{\calo_{X\cap K}(K)} \cong 
	\prod_{\substack {w\in W \\[1pt] \9wL\le K}} 
	\I_{\calo_{X\cap K}(K),\9wL}^A \]
(and vanishes if no conjugate of $L$ is contained in $K$), 
since for each $U\le K$ and $L\le G$, there is a bijection 
	\begin{multline*} 
	\Mor_{\calo(G)}(L,U) = \{U g\,|\,g\in G,~\9gL\le U\} \\
	= \coprod\nolimits_{w\in W} \{U hw\,|\, h\in K, ~\9{hw}L\le U\} 
	\cong \coprod\nolimits_{w\in W} \Mor_{\calo(K)}(\9wL,U) 
	\end{multline*}
that sends $[g]\in\Mor_{\calo(G)}(L,U)$, for $g\in Kw$, to 
$[gw^{-1}]\in\Mor_{\calo(G)}(\9wL,U)$. This proves \eqref{e:Rqk(FM)-2}.

Fix some $\4K=K/H$ in $Y$. Consider the functor 
	\[ \lambda\:\calo_{X\cap K}(K) \Right3{} \5\chi{\downarrow}\4K \] 
that sends an object $L$ in $\calo_{X\cap K}(K)$ to the object 
$(L,\4L\xto{[1]}\4K)$ and similarly for morphisms. Then $\lambda$ is 
injective (split by the forgetful functor). Its image is a full 
subcategory, since for each morphism 
	\[ (L_1,\4L_1\xto{[1]}\4K) \Right3{[g]} 
	(L_2,\4L_2\xto{[1]}\4K) \]
in $\5\chi{\downarrow}\4K$ between objects of $\Im(\lambda)$, we have 
$[1]\circ[\chi(g)]=[1]$ in 
$\Mor_{\calo(\4G)}(\4L_1,\4K)$, and hence $g\in K$ and 
$[g]\in\Mor_{\calo(K)}(L_1,L_2)$. Each object in 
$\5\chi{\downarrow}\4K$ has the form 
$(L,\4L\xto{[\chi(g)]}\4K)$ for some $L\in X$ and 
$g\in G$ such that $\9gL\le K$, and this pair is 
isomorphic to the object 
$(\9gL,\4{\9gL}\xto{[1]}\4K)=\lambda(\9gL)$. So $\lambda$ 
is an equivalence of categories. 

Thus for each $n\ge0$, 
	\begin{align*} 
	\5\chi^R_*(I_n)(\4K) &= \lim_{\5\chi{\downarrow}\4K} 
	\bigl( (\5\chi{\downarrow}\4K)\op \Right2{\mu\op} \calo_X(G)\op 
	\Right2{I_n} R\mod \bigr) \\
	&\cong \lim_{\calo_{X\cap K}(K)} 
	\bigl( \calo_{X\cap K}(K)\op \xto{~\lambda\op~} 
	(\5\chi{\downarrow}\4K)\op \xto{~\mu\op~} \calo_X(G)\op 
	\xto{~I_n~} R\mod \bigr) \\
	&= \lim_{\calo_{X\cap K}(K)}(I_n|_{\calo_{X\cap K}(K)}),
	\end{align*}
where the isomorphism holds by Lemma \ref{l:lim*(C0)} and since each 
undercategory $x\dn\lambda$ has an initial object (hence $(x\dn\lambda)\op$ 
is directed). So 
	\begin{multline*} 
	(R^*\5\chi^R_*)(\Phi)(\4K) \cong H^*\bigl(0\to \5\chi^R_*(I_0)(\4K) \too 
	\5\chi^R_*(I_1)(\4K)\too \cdots\bigr) \\
	\cong H^*\bigl( 0 \to \lim(I_0|_{\calo_{X\cap K}(K)}) \too 
	\lim(I_1|_{\calo_{X\cap K}(K)}) \too\cdots \bigr) \cong 
	\higherlim{\calo_{X\cap K}(K)}*(\Phi|_{\calo_{X\cap K}(K)}), 
	\end{multline*}
where the last isomorphism holds since each $I_n|_{\calo_{X\cap K}(K)}$ 
is injective. 

Now assume that $K_1/H,K_2/H\in Y$, and $g\in G$ is such that $\9gK_1\le 
K_2$. We must identify, for each $n\ge0$, the homomorphism
	\[ [g]^* \: \lim_{\calo_{X\cap K_2}(K_2)}(I_n|_{\calo_{X\cap K_2}(K_2)}) 
	\Right3{} 
	\lim_{\calo_{X\cap K_1}(K_1)}(I_n|_{\calo_{X\cap K_1}(K_1)}) \]
induced by $[g]\in\Mor_{\calo_Y(G/H)}(K_1/H,K_2/H)$. Fix an element 
$y=(y_L)_{L\in X\cap K_2}$ in the first inverse limit; thus $y_L\in I_n(L)$ 
for each $L$. This extends to an element 
	\[ \5y=(\5y_{(L,\4L\xto{[h]}\4K_2)}) \in 
	\lim_{\5\chi\dn\4K_2}(I_n\circ\mu\op) \qquad \textup{where}\qquad 
	\5y_{(L,\4L\xto{[h]}\4K_2)}=I_n([h])(y_{\9xL}), \]
and $[g]$ sends $\5y$ to 
	\[ \5z = (\5z_{(L,\4L\xto{[h]}\4K_1)}) \qquad\textup{where}\qquad 
	\5z_{(L,\4L\xto{[h]}\4K_1)} = \5y_{(L,\4L\xto{[gh]}\4K_2)}. \]
Then $[g]^*(y_L)_{L\in X\cap K_2}=(z_L)_{L\in X\cap K_1}$, where for each 
$L\in X\cap K_1$, 
	\[ z_L = \5z_{(L,\4L\xto{[1]}\4K_1)} = \5y_{(L,\4L\xto{[g]}\4K_2)} = 
	I_n([g])(y_{(\9gL)}). \]
Thus $[g]^*$ is induced by $c_g^{K_1}\:\calo_{X\cap 
K_1}(K_1)\too\calo_{X\cap K_2}(K_2)$, together with the natural 
transformation of functors induced by $[g]\in\Mor_{\calo_X(G)}(L,\9gL)$ for 
$L\in X\cap K_1$.
\end{proof}

\endgroup

Note that the Lyndon-Hochschild-Serre spectral sequence in cohomology (see, 
e.g., \cite[Theorem 6.8.2]{Weibel}) for the group extension $1\too H\too 
G\too G/H\too1$ is the special case of the spectral sequence of Theorem 
\ref{t:G/H} when $X$ and $Y$ both contain only the trivial subgroup. Other 
special cases of this spectral sequence will be looked at in the next 
section.


\section{\texorpdfstring{$\Lambda$}{Lambda}-functors}
\label{s:Lambda}

We now restrict attention to certain functors $\Lambda^*(G;M)$, defined 
originally in \cite[Definition 5.3]{JMO} when $G$ is a finite group. Most 
of our main results, including Theorems \ref{ThA}--\ref{ThD} in the 
introduction, are proven in this section.

\begin{Defi} \label{d:Lambda}
Fix a prime $p$. Let $G$ be a group, and let $M$ be a $\Z G$-module. Define 
a functor $F_M\:\calo_p(G)\op\too\Ab$ by setting 
	\[ F_M(P) = \begin{cases} 
	M & \textup{if $P=1$} \\
	0 & \textup{if $P\ne1$}
	\end{cases} \]
for $P$ a $p$-subgroup of $G$, where $\Aut_{\calo_p(G)}(1)\cong G$ has 
the given action on $M$. Set 
	\[ \Lambda^*(G;M) = \higherlim{\calo_p(G)}*(F_M) 
	\qquad\textup{and}\qquad   \Lambda_f^*(G;M) = 
	\higherlim{\calo_p^f(G)}*(F_M|_{\calo_p^f(G)}). \]
More generally, if $X$ is a set of $p$-subgroups of $G$ invariant under 
conjugation, set 
	\[ \Lambda^*_X(G;M)=\higherlim{\calo_X(G)}*(F_M|_{\calo_X(G)}). \]
\end{Defi}

Note that the functors $\Lambda^*(-;-)$ depend on the prime $p$, even 
though that has been left out of the notation to keep it simple. 

The importance of these graded groups comes from the following proposition, 
which was shown in the finite case in \cite[Lemma 5.4]{JMO}, and in a more 
general situation in \cite[Lemma 5.10]{BLO3}. (In fact, it was formulated 
in \cite[Lemma 5.4]{JMO} in a way that held more generally for compact Lie 
groups.)

\begin{Prop} \label{p:lim*=Lbda}
Let $p$ be a prime, and let $G$ be a group. 
Let $X$ be a set of $p$-subgroups of $G$ invariant under conjugation, fix 
$Q\in X$, and assume that either 
\begin{enumi} 
\item each member of $X$ contains an abelian subgroup of finite index; or 
\item each member of $X$ that contains $Q$ contains it with finite index; or 
\item for each $Q<P\in X$, we have $Q<N_P(Q)$. 
\end{enumi}
Assume also that $Q\le P\in X$ implies $N_P(Q)\in X$. Let $Y$ be the set of 
all subgroups $P/Q$ for $Q\nsg P\in X$. Then for each 
$\Phi\:\calo_X(G)\op\too\Ab$ such that $\Phi(P)=0$ for all $P\in X$ not 
$G$-conjugate to $Q$, 
	\beqq \higherlim{\calo_X(G)}*(\Phi) \cong 
	\Lambda^*_Y(N_G(Q)/Q;\Phi(Q)) .
	\label{e:lim*=Lbda} \eeqq
\end{Prop}


\begin{proof} Since (i) and (ii) imply (iii) by Lemma \ref{l:Q<NP(Q)}, we 
assume from now on that (iii) holds. Set $\Gamma=N_G(Q)/Q$. We want to 
apply \cite[Proposition 5.3]{BLO3} with $\calc=\calo_X(G)$. Let 
$\alpha\:\calo_Y(\Gamma)\too\calo_X(G)$ be the functor that sends an object 
$P/Q\in Y$ to $P\in X$, and sends a morphism $P/Q\xto{[gQ]}R/Q$ to 
$P\xto{[g]}R$ for $g\in N_G(Q)$ such that $\9gP\le R$. To prove the lemma, 
we must show that conditions (a)--(d) in \cite[Proposition 5.3]{BLO3} all 
hold. 

\noindent\textbf{(a) } By definition, $\alpha$ sends 
$\Aut_{\calo_Y(\Gamma)}(1)\cong \Gamma$ isomorphically to 
$\Aut_{\calo_X(G)}(Q)$. 

\noindent\textbf{(b) } We must show, for all $P\in X$ not isomorphic to $Q$ 
in $\calo_X(G)$, that all isotropy subgroups of the action of 
$\Gamma=N_G(Q)/Q$ on $\Mor_{\calo_X(G)}(Q,P)$ are nontrivial and in $Y$. To 
see this, fix $[x]\in\Mor_{\calo_X(G)}(Q,P)$. For $g\in N_G(Q)$, the class 
$[g]\in\Gamma$ is in the isotropy subgroup if and only if $Pxg=Px$, 
equivalently, $xgx^{-1}\in P$. Thus the isotropy subgroup is $(P^x\cap 
N_G(Q))/Q=N_{P^x}(Q)/Q$, and this is nontrivial by (iii) (and is in $Y$ by 
definition of $Y$). 

\noindent\textbf{(c) } Since all morphisms in $\calo_X(G)$ are 
epimorphisms by Lemma \ref{l:epi}, this holds automatically.

\noindent\textbf{(d) } We must show, for all $P/Q\in Y$, all $R\in X$, and 
all $[x]\in\Mor_{\calo_X(G)}(Q,R)$ fixed by the action of $P/Q$ on this 
morphism set, that $[x]$ extends to a morphism defined on $P$. We just saw 
in the proof of (b) that $P\le R^x$ since $P/Q$ fixes $[x]$, so $\9xP\le 
R$, and $[x]$ also defines a morphism from $P$ to $R$. 
\end{proof}

For example, the conclusion of Proposition \ref{p:lim*=Lbda} holds if $X$ and 
$Y$ are the sets of all finite $p$-subgroups, or if $X$ is the set of all 
$p$-subgroups of $G$ that contain an abelian subgroup of finite index and 
$Y$ is the set of all $P/Q$ for $Q\nsg P\in X$.

The following theorem is a special case of Proposition \ref{p:spseq.Phi}, 
and helps reduce computations of $\Lambda^*_f(G;-)$ to the finite case.

\begin{Thm} \label{t:spseq.Lambda} 
Fix a prime $p$. Let $G$ be a locally finite group, and let 
$M$ be a $\Z G$-module. Then there is a first quarter spectral sequence 
	\[ E_2^{ij} \cong \higherlim{\Fin(G)}i \bigl( \Lambda^j(-;M) \bigr)
	\Longrightarrow \Lambda^{i+j}_f(G;M). \] 
If $G$ is countable, then this reduces to short exact sequences
	\[ 0 \Right2{} \higherlim{\Fin(G)}1\!(\Lambda^{i-1}(-;M)) \Right3{} 
	\Lambda^i_f(G;M) \Right3{} \lim_{\Fin(G)}\Lambda^i(-;M) 
	\Right2{} 0 \] 
for each $i\ge0$. 
\end{Thm}

\begin{proof} This follows from Proposition \ref{p:spseq.Phi}, applied with 
$X$ the set of finite $p$-subgroups of $G$ and $\Phi=F_M$. 
\end{proof}

The following result was shown for finite $G$ and $\Z_{(p)}G$-modules in 
\cite[Proposition 6.1(ii)]{JMO}. 

\begin{Prop} \label{p:lim*(OpG)}
Fix a prime $p$, a group $G$, and a $\Z G$-module $M$. 
\begin{enuma} 

\item If $O_p(G)\ne1$, then $\Lambda^*(G;M)=0$.

\item If there is a nontrivial finite normal $p$-subgroup $1\ne Q\nsg G$, 
then $\Lambda^*_f(G;M)=0$.

\item If $G$ is a nontrivial locally finite $p$-group, then 
$\Lambda^*_f(G;M)=0$.

\end{enuma}
\end{Prop}

\begin{proof} \textbf{(a) } Set $Q=O_p(G)$ for short, and let 
$\calo_p^0(G)\subseteq\calo_p(G)$ be the full subcategory with objects the 
$p$-subgroups of $G$ that contain $Q$. 

We apply Lemma \ref{l:X0<X} with $X$ the set of all $p$-subgroups 
of $G$ and $X_0$ the set of those that contain $Q$. Every $p$-subgroup of 
$G$ is contained in one that contains $Q$ since $Q$ is a normal 
$p$-subgroup, and intersections of $p$-subgroups containing $Q$ again 
contain $Q$. Thus conditions (i) and (ii) in Lemma \ref{l:X0<X} hold, and 
hence 
	\[ \higherlim{\calo_p(G)}*(\Phi)\cong 
	\higherlim{\calo_p^0(G)}*(\Phi|_{\calo_p^0(G)}) \]
for each $\calo_p(G)$-module $\Phi$ by Lemma \ref{l:X0<X}. So 
$\Lambda^*(G;M)=0$ since $F_M|_{\calo_p^0(G)}=0$. 

\smallskip

\noindent\textbf{(b) } Now apply Lemma \ref{l:X0<X} with $X$ the set of 
finite $p$-subgroups of $G$ and $X_0$ the set of those that contain $Q$. 
Since $Q$ is finite, each member of $X$ is contained in a member of $X_0$. 
The rest of the proof goes through exactly as in the proof of (a).

\smallskip

\noindent\textbf{(c) } Assume $G\ne1$ is a locally finite $p$-group. For 
each nontrivial finite subgroup $1\ne Q\le P$, $\Lambda^*(Q;M)=0$ by (a). 
Also, $\Lambda^i(1;M)=H^i(1;M)=0$ for $i>0$, while $\Lambda^0(1;M)=M$. So 
$E^{ij}_2=0$ for all $j>0$ in the spectral sequence of Theorem 
\ref{t:spseq.Lambda}. Hence by that spectral sequence, 
$\Lambda^i_f(G;M)\cong E^{i0}_2\cong\higherlim{\Fin(G)}i(\Psi_M)$ for all 
$i\ge0$, where $\Psi_M\:\Fin(G)\too\Ab$ sends the trivial subgroup to $M$ 
and all others to $0$. 

Fix a nontrivial finite subgroup $1\ne R\le G$, and let $\Fin^0(G)$ be the 
set of finite subgroups of $G$ that contain $R$. Then each $H\in\Fin(G)$ is 
contained in $\gen{H,R}\in\Fin^0(G)$ since $G$ is locally finite. Also, if 
$H\in\Fin(G)$ is contained in $K_1,K_2\in\Fin^0(G)$, then $K_1\cap 
K_2\in\Fin^0(G)$ also contains $H$. So the categories 
$\Fin^0(G)\subseteq\Fin(G)$ satisfy the hypotheses of Lemma 
\ref{l:lim*(C0)}, and hence 
	\[ \higherlim{\Fin(G)}*(\Psi_M) \cong 
	\higherlim{\Fin^0(G)}*(\Psi_M|_{\Fin^0(G)}) = 0 \]
since $\Psi|_{\Fin^0(G)}=0$.
\end{proof}

When $G$ is locally finite, $\Phi=F_M$ for some $\Z G$-module $M$, and $X$ 
and $Y$ are the sets of all $p$-subgroups of $G$ and $G/H$, the spectral 
sequence of Theorem \ref{t:G/H} takes the following form:

\begin{Thm} \label{t:Lb.sp.seq}
Fix a prime $p$, a locally finite group $G$, a normal subgroup $H\nsg G$, 
and a $\Z G$-module $M$. Then there is a first quarter spectral sequence 
	\[ E_2^{ij} = \higherlim{\calo_p(G/H)}i\bigl( P/H \mapsto 
	\Lambda^j(P;M) \bigr) 
	\Longrightarrow \Lambda^{i+j}(G;M) . \] 
\end{Thm}

Another special case of Theorem \ref{t:G/H} is the following:

\begin{Thm} \label{t:Lb=Lbf}
Fix a prime $p$, and assume $G$ is a group all of whose $p$-subgroups are 
locally finite. Then for each $\Z G$-module $M$, the natural homomorphism 
$\Lambda^*(G;M)\too\Lambda^*_f(G;M)$ is an isomorphism. 
\end{Thm}

\begin{proof} We again apply Theorem \ref{t:G/H}, this time with $H=1$ the 
trivial subgroup, $X$ the set of finite $p$-subgroups of $G$, and $Y$ the 
set of all $p$-subgroups. For each $\calo_p(G)$-module $\Phi$, the spectral 
sequence takes the form 
	\[ E_2^{ij} = \higherlim{\calo_p(G)}i
	\Bigl(K\mapsto\higherlim{\calo_p^f(K)}j(\Phi|_{\calo_p^f(K)}) \Bigr) 
	\implies \higherlim{\calo_p^f(G)}{i+j}(\Phi). \]
When $\Phi=F_M$, this yields a spectral sequence
	\beqq E_2^{ij} = \higherlim{\calo_p(G)}i(\Lambda^j_f(-;M)) 
	\implies \Lambda^{i+j}_f(G;M). \label{e:LbLbf} \eeqq

If $K\ne1$ is a nontrivial $p$-subgroup of $G$, then $\Lambda^*_f(K;M)=0$ 
by Proposition \ref{p:lim*(OpG)}(c) and since $K$ is locally finite. If 
$K=1$, then $\Lambda^*_f(K;M)\cong H^*(K;M)$ is $M$ in degree $0$ and zero 
in higher degrees. Thus $\Lambda^0_f(-;M)\cong F_M|_{\calo_p^f(G)}$, while 
$\Lambda^j_f(-;M)=0$ (as a $\calo_p^f(G)$-module) for $j>0$. So by 
\eqref{e:LbLbf}, for each $i\ge0$, 
	\beq \Lambda^i_f(G;M) \cong E_2^{i0} \cong 
	\higherlim{\calo_p(G)}i(F_M) = \Lambda^i(G;M). \qedhere \eeq
\end{proof}

We are now ready to prove \cite[Lemma 5.12]{BLO3}, and through that finish 
the proof of \cite[Theorem 8.7]{BLO3}. In Step 1 of the argument given in 
\cite{BLO3}, we claimed that certain injective functors $\mathfrak{I}_k$ 
are inverse systems of groups and surjections. Unfortunately, the maps in 
the inverse systems need not be surjective, which means that we can get 
into trouble with nonvanishing ${\lim}^1(-)$. This is due to the fact that if 
$Q$ and $P_0<P$ are $p$-subgroups of a locally finite group $G$, then the 
natural map $\Mor_{\calo_p(G)}(Q,P_0) \too \Mor_{\calo_p(G)}(Q,P)$ need not 
be injective. 

Instead, we use the following argument, based on Theorems 
\ref{t:spseq.Lambda} and \ref{t:Lb=Lbf}.

\begin{Thm} \label{t:5.12}
Fix a prime $p$. 
Let $G$ be a locally finite group, and let $M$ be a $\Z G$-module. Assume, 
for some finite subgroup $H_0\le G$, that $\Lambda^*(H;M)=0$ for each 
finite subgroup $H\le G$ that contains $H_0$. Then $\Lambda^*(G;M)=0$. In 
particular, this holds if $M$ is a $\Z_{(p)}G$-module and $C_G(M)$ contains 
at least one element of order $p$. 
\end{Thm}

\begin{proof} Let $\Fin_0(G)$ be the set of finite subgroups of $G$ that 
contain $H_0$. Thus by assumption, $\Lambda^*(K;M)=0$ for all 
$K\in\Fin_0(G)$. Since $G$ is locally finite, each finite subgroup 
$K\in\Fin(G)$ is contained in $\gen{K,H_0}\in\Fin_0(G)$. Also, the 
intersection of two members of $\Fin_0(G)$ again lies in the set. Since 
there is at most one morphism between any pair of objects in the category 
$\Fin(G)$, this proves that opposite categories of undercategories for the 
inclusion $\cali\:\Fin_0(G)\too\Fin(G)$ are nonempty and directed, and 
hence by Lemma \ref{l:lim*(C0)} that 
	\[ \higherlim{\Fin(G)}*(\Lambda^*(-;M)) \cong 
	\higherlim{\Fin_0(G)}*(\Lambda^*(-;M)|_{\Fin_0(G)}) = 0. \]
Hence $\Lambda^*_f(G;M)=0$ by the spectral sequence of 
Theorem \ref{t:spseq.Lambda}, and so $\Lambda^*(G;M)=0$ by Theorem 
\ref{t:Lb=Lbf}.

If $M$ is a $\Z_{(p)}G$-module and there is $g\in C_G(M)$ of order $p$, 
then $\Lambda^*(H;M)=0$ for each $H\in\Fin(G)$ containing $g$ by 
\cite[Proposition 6.1(ii)]{JMO}. So $\Lambda^*(G;M)=0$ by the first part of 
the statement, applied with $H_0=\gen{g}$.
\end{proof}

Lemma 5.12 of \cite{BLO3} was applied only once in that paper: in the proof 
of Theorem 8.7 when we showed that $|\call_S^c(G)|\pcom\simeq BG\pcom$ for 
an LFS(p)-group $G$. We restate the full theorem here, but it is only the 
last statement in it ($|\call_S^c(G)|\pcom\simeq BG\pcom$) that was 
affected by the error.

\begin{Thm}[{\cite[Theorem 8.7]{BLO3}}] \label{t:blo3-8.7}
Let $G$ be a locally finite group all of whose $p$-sub\-groups are discrete 
$p$-toral. Assume in addition that for each increasing sequence $A_1\le 
A_2\le\cdots$ of finite abelian $p$-subgroups of $G$, there is $k\ge1$ such 
that $C_G(A_n)=C_G(A_k)$ for all $n\ge k$. Then $G$ has a unique conjugacy 
class $\sylp{G}$ of maximal $p$-subgroups, and for each $S\in\sylp{G}$, the 
triple $(S,\calf_S(G),\call_S^c(G))$ is a $p$-local compact group with 
classifying space $|\call_S^c(G)|\pcom\simeq BG\pcom$.
\end{Thm}

\section{Examples} \label{s:examples}

We finish with a few examples that show how higher limits over orbit 
categories of locally finite groups can differ from those over orbit 
categories of finite groups. The following groups will be used whenever we 
need a more concrete example to see this. When $G_1,G_2,\dots$ is a 
sequence of groups, we write $\bigoplus_{i=1}^\infty G_i$ to mean the group 
of elements of finite support in the direct product of the $G_i$.
\beqq \parbox{145mm}
{Fix a prime $p$. Let $\F_0$ be a finite field of characteristic $p$ and 
order at least $3$, and let $\F\supseteq\F_0$ be the extension of degree 
$p$. Choose $1\ne U\le\F^\times$ such that $U\cap\F_0^\times=1$. Let 
$S\le\Aut(\F)$ be the subgroup of order $p$ (so $\Fix{S}{\F}=\F_0$). Set 
\begin{center} 
	$\renewcommand{\arraycolsep}{15pt}
	\begin{array}{r@{\,\,=\,\,}lr@{\,\,=\,\,}lr@{\,\,=\,\,}l}
	H & \displaystyle\bigoplus\nolimits_{i=1}^\infty \F^\times, &
	H_0 & \displaystyle\bigoplus\nolimits_{i=1}^\infty U \le H, &
	M & \displaystyle\bigoplus\nolimits_{i=1}^\infty \F \\[2mm]
	\Gamma & H\rtimes S, & 
	\Gamma_0 & H_0\rtimes S \le \Gamma, & 
	\Gamma_* & (U\times H)\rtimes S. 
	\end{array} $
\end{center}
Let $U$ and $H$ act on $M$ by setting $u(x_1,x_2,\dots)=(ux_1,ux_2,\dots)$ 
and $(h_1,h_2,\dots)(x_1,x_2,\cdots) = (h_1x_1,h_2x_2,\dots)$,
and let $\Gamma$ and $\Gamma_*$ act on $M$ via those actions 
and the Galois action of $S$.} \label{e:HGM} \eeqq

We first note the following very general vanishing result:

\begin{Lem} \label{l:Lb+rk}
Fix a prime $p$, and let $G$ be a countable locally finite group with 
finite Sylow $p$-subgroup $S\le G$ of order $p^n$ (some $n\ge0$). Then 
for each $\Z_{(p)}G$-module $M$, $\Lambda^i(G;M)=0$ for all $i\ge n+2$. 
\end{Lem}

\begin{proof} For each $K\in\Fin(G)$, and each $i\ge1$ such that 
$\Lambda^i(K;M)\ne0$, we have $i\le n$ by \cite[Lemma III.5.27]{AKO}.  
So for each $i\ge n+2$, $\Lambda^i(K;M)\cong 
\Lambda^{i-1}(K;M)=0$ for all $K\in\Fin(G)$, and hence $\Lambda^i(G;M)= 
\Lambda^i_f(G;M)=0$ by Theorem \ref{t:spseq.Lambda}. 
\end{proof}

By a result of Jackowski and McClure (see \cite[Proposition 5.2]{JMO}), for 
any finite group $G$ and any $\Z_{(p)}G$-module $M$, the functor 
$P\mapsto\Fix{P}M$ is acyclic. The following example, an application of 
the spectral sequence of Proposition \ref{p:spseq.Phi}, shows that this 
does not hold in general for locally finite groups. 

\begin{Ex} \label{ex:P->M^P}
Fix a prime $p$, let $G$ be a locally finite group, and let $M$ be a 
(left) $\Z_{(p)}G$-module. Consider the $\Z_{(p)}\calo_p^f(G)$-module 
	\[ \Phi^G_M \: \calo_p^f(G)\op \Right4{} \Z_{(p)}\mod \]
defined by setting $\Phi^G_M(P)=\Fix{P}M$ (the $P$-invariant elements of $M$) for 
each finite $p$-subgroup $P\le G$, and by sending a morphism $P\xto{[g]}Q$ 
to the composite 
	\[ \Fix{Q}M \Right4{\incl} \Fix{\9gP}M \Right4{g^{-1}} \Fix{P}M. \]
Then 
	\begin{align} 
	\higherlim{\calo_p^f(G)}i&(\Phi^G_M) \cong 
	\higherlim{\Fin(G)}i(K\mapsto\Fix{K}M) \label{e:P->M^P-1} \\
	&\cong \begin{cases} \Fix{G}M & \textup{if $i=0$} \\[2pt] 
	\Coker\bigl[M\too \bigl(\lim\limits_{\Fin(G)}(M/\Fix{-}M)\bigr)\bigr] 
	& \textup{if $i=1$} \\
	0 & \dbl{\textup{if $|G|\le\aleph_m$ and $i\ge 
	m+2$}}{\textup{(some $m\ge0$).}}
	\end{cases} \label{e:P->M^P-2}
	\end{align}
\end{Ex}

\begin{proof} By \cite[Proposition 5.2 and Corollary 1.8]{JMO}, for each 
$K\in\Fin(G)$, we have that $\lim^0(\Phi^G_M|_{\calo_p(K)})\cong \Fix{K}M$, 
while $\lim^i(\Phi^G_M|_{\calo_p(K)})=0$ for all $i>0$. So the spectral 
sequence of Proposition \ref{p:spseq.Phi} collapses, and \eqref{e:P->M^P-1} 
follows directly from that. If $|G|\le\aleph_m$ where $0\le m<\infty$, then 
$|\Fin(G)|\le\aleph_m$, and so $\lim_{\Fin(G)}^i(\Psi)=0$ for each functor 
$\Psi$ and each $i\ge m+2$ by a theorem of Goblot \cite[Proposition 
2]{Goblot} (see also \cite[Th\'eor\`eme 3.1]{Jensen}). 

The constant functor on $\Fin(G)$ with value $M$ is isomorphic to $\I_1^M$, 
and hence is acyclic by Proposition \ref{p:I_c^A}(c). So the 
short exact sequence 
	\[ 0\too \Fix{-}M\too M\too M/\Fix{-}M\too0 \]
of $\Z_{(p)}\Fin(G)$-modules induces an exact sequence 
	\begin{multline*} 
	0 \too \higherlim{\Fin(G)}0(\Fix{-}M) \Right2{} M 
	\Right2{} \higherlim{\Fin(G)}0(M/\Fix{-}M) \Right2{} 
	\\ \higherlim{\Fin(G)}1(\Fix{-}M) \too 0, 
	\end{multline*}
and $\lim_{\Fin(G)}^i(\Fix{-}M)$ is as described in \eqref{e:P->M^P-2} for 
$i=0,1$. 
\end{proof}

For example, if we return to the pair $(\Gamma,M)$ defined in 
\eqref{e:HGM}, we get 
	\[ \higherlim{\calo_p(\Gamma)}0(\Phi^\Gamma_M)=0 \qquad\textup{and}\qquad
	\higherlim{\calo_p(\Gamma)}1(\Phi^\Gamma_M)\cong 
	\Bigl(\prod\nolimits_{i=1}^\infty\F_0\Bigr)\Big/
	\Bigl(\bigoplus\nolimits_{i=1}^\infty \F_0\Bigr) \ne0, \]
Hence $\Phi^\Gamma_M$ is not acyclic as an $\calo_p(\Gamma)$-module.
(Note that $\calo_p^f(\Gamma)=\calo_p(\Gamma)$ in this case.)

In the situation of Example \ref{ex:P->M^P}, if $G$ has no elements of 
order $p$, then $\calo_p(G)=\calo_p^f(G)$ has only one object, so an 
$\calo_p(G)$-module is the same as a $\Z G$-module, and 
$\Lambda^*(G;M)\cong H^*(G;M)$ for each such module $M$. This situation has 
been studied by Holt, who showed in \cite[Theorem 1]{Holt1} that for each 
locally finite group $G$ with no $p$-torsion satisfying a certain condition 
($*$) (which is always satisfied if $G$ is abelian), if $|G|=\aleph_n$ for 
some $n\ge0$, then $H^{n+1}(G;M)\ne0$ for some $\F_pG$-module $M$. In 
particular, in the situation of Example \ref{ex:P->M^P}, this shows that 
$\lim_{\calo_p^f(G)}^n(\Phi^G_M)$ can be nonzero for arbitrarily large $n$, 
while in the situation of Proposition \ref{p:spseq.Phi}, $E_2^{ij}$ can be 
nonzero for arbitrarily large $i$. 

The pair $(H,M)$ in \eqref{e:HGM} gives an example in the countable case of 
the type studied by Holt: Example \ref{ex:P->M^P} implies that 
	\beqq H^1(H;M) \cong \Lambda^1(H;M) \cong 
	\Bigl(\prod\nolimits_{i=1}^\infty\F\Bigr)\Big/
	\Bigl(\bigoplus\nolimits_{i=1}^\infty \F\Bigr) \ne0 . 
	\label{e:H1(H;M)} \eeqq

Many of the basic properties of the functors $\Lambda^*(G;M)$ for finite 
$G$ were listed in \cite[Propositions 6.1--6.2]{JMO}. For example, 
\cite[Proposition 6.2(i)]{JMO} states that if $G$ is a finite group, $M$ is 
a $\Z_{(p)}G$-module, and $S\in\sylp{G}$ has order $p$, then 
	\[ \Lambda^1(G;M)\cong\Fix{N_G(S)}M/\Fix{G}M, \]
and $\Lambda^i(G;M)=0$ for $i\ne1$. Neither of these holds in general when 
$G$ is locally finite and infinite.

\begin{Ex} \label{ex:|S|=p}
Fix a prime $p$. Let $G$ be a countable locally finite group that contains 
a maximal $p$-subgroup $S\le G$ of order $p$, and let $\Fin_0(G)$ be the 
poset of finite subgroups of $G$ that contain $S$. Then for each 
$\Z_{(p)}G$-module $M$, we have 
	\begin{align} 
	\Lambda^1(G;M) &\cong \lim_{\Fin_0(G)}\Bigl(K\mapsto 
	\tfrac{\Fix{N_{K}(S)}M}{\Fix{K}M)} \Bigr) \label{e1:|S|=p} \\
	\Lambda^2(G;M) &\cong \Coker\Bigl[ 
	\lim_{\Fin_0(G)}\Bigl(K\mapsto\tfrac{\Fix{S}M}{\Fix{K}M}\Bigr) 
	\Right3{} 
	\lim_{\Fin_0(G)}\Bigl(K\mapsto\tfrac{\Fix{S}M}{\Fix{N_K(S)}M}\Bigr) 
	\Bigr] \label{e2:|S|=p} \end{align}
and $\Lambda^i(G;M)=0$ for all $i\ne1,2$. 
\end{Ex}

\begin{proof} If $P\le G$ is a finite $p$-subgroup, then $\gen{S,P}$ is 
finite, and $S\in\sylp{\gen{S,P}}$ since it is a maximal $p$-subgroup of 
$G$. Thus $|P|\le p$. In particular, every $p$-subgroup of $G$ is finite, 
so $\Lambda^*(G;M)=\Lambda^*_f(G;M)$.

For each $K\in\Fin_0(G)$, we have $S\in\sylp{K}$ since it is a maximal 
$p$-subgroup. So by \cite[Proposition 6.2(i)]{JMO}, $\Lambda^1(K;M)\cong 
\Fix{N_K(S)}M/\Fix{K}M$, and $\Lambda^i(K;M)=0$ for all $i\ne1$. Theorem 
\ref{t:spseq.Lambda} now says that 
	\[ \Lambda^i(G;M)\cong 
	\higherlim{K\in\Fin_0(G)}{i-1}\hskip-4mm\bigl(\Fix{N_K(S)}M/\Fix{K}M
	\bigr) \] 
when $i=1,2$, while $\Lambda^i(G;M)=0$ for $i\ne1,2$. This proves 
\eqref{e1:|S|=p} and the last statement.

Upon applying limits to the short exact sequence 
	\[ 0 \Right2{} \Fix{K}M \Right3{} \Fix{N_K(S)}M \Right3{} 
	\Fix{N_K(S)}M/\Fix{K}M \Right2{} 0 \]
of $\Fin_0(G)$-modules, we get an exact sequence 
	\beqq \higherlim{K\in\Fin_0(G)}1(\Fix{K}M) 
	\Right2{} \higherlim{K\in\Fin_0(G)}1(\Fix{N_K(S)}M) \Right2{} 
	\Lambda^2(G;M) \Right2{} 0 \label{e3:|S|=p} \eeqq
(recall that $\Fin(G)$ is a countable directed poset). For 
$T=1$ or $T=S$, the extension 
	\[ 0\to \Fix{N_K(T)}M\to \Fix{S}M\to \Fix{S}M/\Fix{N_K(T)}M\to 0 \] 
of $\Fin_0(G)$-modules induces an isomorphism 
	\beqq \higherlim{K\in\Fin_0(G)}1(\Fix{N_K(T)}M) \cong \Coker \Bigl[ 
	\Fix{S}M \Right2{} \lim_{K\in\Fin_0(G)} 
	\Bigl(\tfrac{\Fix{S}M}{\Fix{N_K(T)}M}\Bigr)
	\Bigr] \label{e4:|S|=p} \eeqq
and \eqref{e2:|S|=p} follows from \eqref{e3:|S|=p} and 
\eqref{e4:|S|=p}.
\end{proof}

Upon returning to the groups and modules defined in 
\eqref{e:HGM}, formulas \eqref{e1:|S|=p} and \eqref{e2:|S|=p} imply 
	\beqq \Lambda^1(\Gamma_0;M) \cong \prod\nolimits_{i=1}^\infty \F_0 
	\qquad\textup{and}\qquad 
	\Lambda^2(\Gamma_*;M) \cong \Bigl(\prod\nolimits_{i=1}^\infty \F_0\Bigr) 
	\Big/ \Bigl(\bigoplus\nolimits_{i=1}^\infty \F_0\Bigr) \ne0 . 
	\label{e5:|S|=p} \eeqq
In contrast, $\Fix{N_{\Gamma_0}(S)}M/\Fix{\Gamma_0}M\cong 
\bigoplus_{i=1}^\infty \F_0$. So when $\Gamma$ is locally finite with Sylow 
$p$-subgroups of order $p$, then $\Lambda^1(\Gamma;M)$ and 
$\Lambda^2(\Gamma;M)$ can both be larger than would be predicted by the 
formulas for finite $\Gamma$.

When $G$ is a product of two locally finite groups, Theorem 
\ref{t:Lb.sp.seq} takes the form:

\begin{Prop} \label{p:GxH}
Let $G_1$ and $G_2$ be locally finite groups, and let $M$ be a $\Z[G_1\times 
G_2]$-module. Then there is a first quarter spectral sequence 
	\[ E_2^{ij} = \Lambda^i(G_1;\Lambda^j(G_2;M)) \Longrightarrow 
	\Lambda^{i+j}(G_1\times G_2;M). \]
Here, $G_1$ acts on $\Lambda^*(G_2;M)$ via its action on $M$. 
\end{Prop}

\begin{proof} By Theorem \ref{t:Lb.sp.seq}, there is a spectral sequence 
converging to $\Lambda^*(G_1\times G_2;M)$ with 
$E_2^{ij}\cong\lim_{\calo_p(G_1)}^i(\Psi^j)$, where for a $p$-subgroup $P\le 
G_1$, we have $\Psi^j(P)=\Lambda^j(P\times G_2;M)$. When $P\ne1$, we have 
$O_p(P\times G_2)\ne1$, and so $\Lambda^j(P\times G_2;M)=0$ by Proposition 
\ref{p:lim*(OpG)}(a). So $\Psi^j=F_{\Lambda^j(G_2;M)}$, and $E_2^{ij}$ is as 
described above.
\end{proof}

By \cite[Proposition 6.1(v)]{JMO}, when $H$ and $K$ are finite groups, $M$ 
is an $\F_pH$-module, and $N$ is an $\F_pK$-module, then $\Lambda^*(H\times 
K;M\otimes_{\F_p}N)$ is isomorphic to the tensor product of 
$\Lambda^*(H;M)$ with $\Lambda^*(K;N)$. (A more general K\"unneth formula 
is given there.) This is not in general the case for products of locally finite 
groups, since tensor products don't in general commute with taking limits 
(and also since higher limits are involved). For example, 
if $(\Gamma_*,M)$ and $(H_0,M)$ are as in \eqref{e:HGM}, then by Lemma 
\ref{l:Lb+rk}, \eqref{e:H1(H;M)}, and \eqref{e5:|S|=p}, 
	\beq \begin{split} 
	\Lambda^3(\Gamma_*\times H;M\otimes_{\F_p}M)=0 \qquad 
	&\textup{while}\qquad 
	\Lambda^2(\Gamma_*;M)\otimes_{\F_p}\Lambda^1(H;M)\ne0 \\
	\Lambda^4(\Gamma_*\times \Gamma_*;M\otimes_{\F_p}M)=0 \qquad 
	&\textup{while}\qquad 
	\Lambda^2(\Gamma_*;M)\otimes_{\F_p}\Lambda^2(\Gamma_*;M)\ne0 .
	\end{split} \eeq


\appendix

\section{Locally finite \texorpdfstring{$p$}{p}-groups}
\label{s:appx}

In the proof of Proposition \ref{p:lim*=Lbda}, we needed to know, for 
certain pairs $Q<P$ of $p$-groups, that $Q$ is strictly contained in its 
normalizer $N_P(Q)$. This motivates the following lemma, which gives 
conditions under which this holds, and also motivates Example 
\ref{ex:N(Q)=Q} where we show that it is not always the case, not even 
for locally finite $p$-groups.

\begin{Lem} \label{l:Q<NP(Q)}
Let $P$ be a $p$-group, and let $Q<P$ be a proper subgroup. Assume either 
\begin{enumi} 
\item $Q$ has finite index in $P$, or 
\item $P$ contains an abelian subgroup of finite index. 
\end{enumi}
Then $Q<N_P(Q)$.
\end{Lem}

\begin{proof} We first recall the well known fact that 
	\beqq \textup{$H\le G$ groups and $[G:H]<\infty$ $\implies$
	$\exists\,N\nsg G$ with $N\le H$ and $[G:N]<\infty$.}
	\label{e:Q<NP(Q)} \eeqq
This follows upon setting $N=C_{G}(G/H)$, where $G$ acts on the finite set 
$G/H$ via translation of cosets.

Assume first that (i) holds. By \eqref{e:Q<NP(Q)}, there is $N\nsg P$ 
of finite index and contained in $Q$. Then $P/N$ is a finite $p$-group and 
$Q/N<P/N$ is a proper subgroup, so $Q/N<N_{P/N}(Q/N)=N_P(Q)/N$. Thus 
$Q<N_P(Q)$.

Now assume (ii) holds, and let $A\le P$ be an abelian subgroup of finite 
index. By \eqref{e:Q<NP(Q)} again, we can assume that $A\nsg P$. If $Q\ge 
A$, then $[P:Q]<\infty$, and so $Q<N_P(Q)$ by case (i). So assume $Q\ngeq 
A$, choose $a\in A\sminus Q$, and let $\5a$ be its conjugacy class in $P$: 
a finite subset of $A$ since $P/A$ is finite and $A$ is abelian. Then 
$\gen{\5a}$ is a finite normal subgroup of $P$, so $Q$ has finite index in 
$Q\gen{\5a}$, and hence $Q< N_{Q\gen{\5a}}(Q)\le N_P(Q)$ by case (i). 
\end{proof}

The following example shows why some conditions are needed to ensure that 
$N_P(Q)>Q$ for a pair of $p$-groups $P>Q$.

\begin{Ex} \label{ex:N(Q)=Q}
Define inductively a sequence of $p$-groups $P_0,P_1,P_2,\cdots$ with 
subgroups $A_n,Q_n\le P_n$ as follows:
	\begin{align*} 
	P_0&=C_p & Q_0&=1\le P_0 & A_0&=P_0 \\
	P_{n+1}&=P_{n}\wr C_p & Q_{n+1}&=Q_{n}\wr C_p\le P_{n+1} & 
	A_{n+1}&=(A_n)^{\times p} \nsg P_{n+1}. 
	\end{align*}
Here, $(-)^{\times p}$ means the $p$-fold direct product. 
Thus for each $n$, $A_n$ is elementary abelian of rank $p^n$, $P_n$ is an 
$(n+1)$-times iterated wreath product $C_p\wr\cdots\wr C_p$, and 
$A_nQ_n=P_n$ and $A_n\cap Q_n=1$. 

Define injective homomorphisms $\varphi_n\:P_n\too P_{n+1}$ by setting 
	\[ \varphi_n(g) = \begin{cases} 
	(g,1,\dots,1) \in (P_n)^{\times p} \nsg P_{n+1} & 
	\textup{if $n$ is odd} \\
	(g,g,\dots,g) \in (P_n)^{\times p} \nsg P_{n+1} & 
	\textup{if $n$ is even.} 
	\end{cases} \]
Set 
	\[ P = \colim\nolimits_n(P_n,\varphi_n), \qquad
	Q = \colim\nolimits_n(Q_n,\varphi_n|_{Q_n}), 
	\qquad\textup{and}\qquad
	A = \colim\nolimits_n(A_n,\varphi_n|_{A_n}). \] 
Then $P$ is a locally finite $p$-group, and $A,Q\le P$ have 
the following properties:
\begin{enuma} 

\item $Z(P)=1$ and $P$ has no finite nontrivial normal subgroups; 

\item $[P,P]=P$ and $P$ has no proper subgroups of finite index; 

\item $A\nsg P$, $A\cap Q=1$, and $AQ=P$; and 

\item $N_P(Q)=Q$.

\end{enuma}
\end{Ex}

\begin{proof} Let $\psi_n\:P_n\too P$ be the natural map: an injective 
homomorphism since the $\varphi_n$ are all injective. Set 
$P_n^*=\psi_n(P_n)<P$ for short. Thus $P_n^*\cong P_n$, and $P$ is the 
union of the increasing sequence $P_0^*<P_1^*<P_2^*<\cdots$ of finite 
$p$-groups. In particular, $P$ is a locally finite $p$-group. 

\smallskip

\noindent\textbf{(a) } If $z\in Z(P)$, then $z\in P_n^*$ for some $n$, and 
so $z\in Z(P_m^*)$ for all $m\ge n$. Since $|Z(P_m)|=p$ for each $m$, and 
$Z(P_{m+1})\cap \varphi_m(Z(P_m))=1$ if $m$ is odd, we get that $Z(P)=1$. 

Assume $1\ne N\nsg P$ is a finite nontrivial normal subgroup. Then 
$1\ne Z(N)\nsg P$, and $\Aut_P(Z(N))$ is a 
finite $p$-group of automorphisms of $Z(N)$. So 
$1\ne C_{Z(N)}(\Aut_P(Z(N))) \le Z(P)$, which we just saw is impossible.

\smallskip

\noindent\textbf{(b) } When $n$ is even, the subgroup 
$\varphi_n(P_n)<P_{n+1}$ is generated by elements of the form 
$(g,g,\dots,g)$ for $g\in P_n$ of order $p$. For such $g$, 
$(1,g,g^2,\dots,g^{p-1})$ is conjugate in $P_{n+1}$ to 
$(g,g^2,g^3,\dots,g^p)$, so $(g,g,\dots,g)\in[P_{n+1},P_{n+1}]$, and hence 
$P_n^*\le[P,P]$. Thus $P=[P,P]$. 

If $R<P$ is a proper subgroup of finite index, then it contains a normal 
subgroup $N\nsg P$ of finite index by \eqref{e:Q<NP(Q)}, which is 
impossible since that would imply $[P/N,P/N]=[P,P]/N=P/N$ when $P/N$ is a 
nontrivial finite $p$-group. 

\smallskip

\noindent\textbf{(c) } We have $A\cap Q=\bigcup_{n=0}^\infty\psi_n(A_n\cap 
Q_n)=1$ and $AQ\ge\bigcup_{n=0}^\infty\psi_n(A_nQ_n)=P$. For each $g\in P$, 
let $n$ be such that $g\in P_n^*$: then $g$ normalizes $\psi_m(A_m)$ for each 
$m\ge n$ and hence $g$ normalizes $A$. 

\smallskip

\noindent\textbf{(d) } Set $B=N_P(Q)\cap A$. Then $[Q,B]\le Q\cap A=1$, and 
hence $B\le C_P(AQ)=Z(P)$ since $P=AQ$ and $A\ge B$ is abelian. So $B=1$ by (a). 

Each $g\in N_P(Q)$ has the form $g=ah$ for $a\in A$ and $h\in Q$ by (b), 
and $a\in N_P(Q)\cap A=B$. So $N_P(Q)=QB=Q$. 
\end{proof}


\end{document}